\documentclass[11pt]{amsart}

\usepackage{tikz}

\usetikzlibrary{knots}
\usetikzlibrary{decorations.markings}

\usepackage{graphicx} 
\usepackage{t1enc}
\usepackage{amsthm,amssymb}
\usepackage{latexsym}
\usepackage{amsmath} 
\usepackage{graphics}
\usepackage{epsfig,multicol}
\usepackage{amsfonts}
\usepackage[latin1]{inputenc}
\usepackage[english]{babel}
\usepackage{ mathrsfs }

\newtheorem{theorem}{Theorem}
\newtheorem*{theorem*}{Theorem}
\newtheorem*{proposition*}{Proposition}
\newtheorem{proposition}{Proposition}
\newtheorem{lemma}{Lemma}
\newtheorem*{lemma*}{Lemma}

\newtheorem{problem}{Problem}
\newtheorem{corollary}{Corollary}
\newtheorem{remark}{Remark}

\newtheorem{example}{Example}

\def\hpic #1 #2 {\mbox{$\begin{array}[c]{l} \epsfig{file=#1,height=#2}
\end{array}$}}
 
\def\vpic #1 #2 {\mbox{$\begin{array}[c]{l} \epsfig{file=#1,width=#2}
\end{array}$}}

\newcommand  {\rmn}\romannumeral

 \newcommand{\IN}[0]{\mathbb{N}}

 \newcommand{\IZ}[0]{\mathbb{Z}}

 \newcommand{\CL}[0]{\mathcal{L}}

 \newcommand{\CT}[0]{\mathcal{T}}

\begin{document}
\title[On the  oriented Thompson subgroup $\vec{F}_3$ and its relatives in $F_k$]{On the  oriented Thompson subgroup $\vec{F}_3$ and its relatives in higher Brown-Thompson groups}
\author{Valeriano Aiello} 
\address{Valeriano Aiello,
Section de Math\'  ematiques, Universit\' e de Gen\` eve, 2-4 rue du Li\` evre, Case Postale 64,
1211 Gen\` eve 4, Switzerland
}\email{valerianoaiello@gmail.com}
\author{Tatiana Nagnibeda} 
\address{Tatiana Nagnibeda,
Section de Math\'  ematiques, Universit\' e de Gen\` eve, 2-4 rue du Li\` evre, Case Postale 64,
1211 Gen\` eve 4, Switzerland
}\email{tatiana.smirnova-nagnibeda@unige.ch}

\begin{abstract} A few years ago the so-called oriented subgroup $\vec F$ of the Thompson group $F$ was introduced by V.~Jones while investigating the connections between subfactors and conformal field theories. In the coding of links and knots by elements of $F$ it corresponds exactly to the oriented ones. Thanks to the work of Golan and Sapir, $\vec F$ provided the first example of a maximal  subgroup of infinite index in $F$ different from the parabolic subgroups that fix a point in $(0,1)$. In this paper we investigate possible analogues of $\vec F$ in higher Thompson groups $F_k, k\geq 2$, with $F=F_2$, introduced by Brown. Most notably, we study algebraic properties of the oriented subgroup $\vec{F}_3$ of $F_3$, as described recently by Jones, and prove in particular that it gives rise to a non-parabolic maximal subgroup of infinite index in $F_3$ and that the corresponding quasi-regular representation is irreducible.  
\end{abstract}
\maketitle

\section*{Introduction}
Ever since its introduction by Richard Thompson in 1965, the Thompson group $F$ has drawn a great deal of attention 
and striking connections have been found to a wide variety of seemingly different fields such as homotopy theory, 
 logic and cryptography, to name but a few.
Later in 1987 Kenneth Brown \cite{Brown} introduced a family of groups $F_k$, $k\geq 2$, generalizing the Thompson group $F$, with $F_2=F$. Among other things, he showed that these groups are finitely presented and of type FP$_\infty$.
All these groups are groups of piecewise-linear homeomorphisms of the unit interval $[0,1]$.

A few years ago Vaughan Jones \cite{Jo14} discovered a method to construct unitary representations of the Thompson group $F$ by means of planar algebras \cite{jo2}. This construction was later   extended to a broad class of groups including the Brown-Thompson groups thanks to a new categorical framework \cite{Jo16}. 
Several of these representations of the Thompson groups have been studied so far \cite{BJ, ABC, AJ, Jo19}.

In \cite{Jo14} Jones also introduced a procedure which yields unoriented links from elements of the Thompson group $F$. 
Recall that every element of $F$ can be described by a pair of planar rooted binary trees with the same number of leaves. 
Given an element $g=(T_+,T_-)$, 
 the associated knot/link is then denoted by $\CL(g)=\CL(T_+,T_-)$.
In general, the knots/links corresponding to the elements of 
$F$ 
 do not have a natural orientation.
To overcome this issue, Jones introduced 
the so-called oriented subgroup $\vec{F} \leq F$. 
A result analogous to the classical Alexander Theorems holds: given any unoriented (oriented) knot/link $L$ there exists an element $g$ in $F$ (in $\vec{F}$) 
for which $\CL(g)=L$, \cite{Jo14, A}. 
It is worth mentioning that in \cite{ACJ} the subgroup $\vec{F}$ was realized as the group of fractions of the category of oriented forests $\vec{\mathfrak{F}}$ in the sense of  \cite{Jo16}, and
unitary representations associated with the Homflypt polynomial were introduced.
Other representations related to link and graph invariants   were discussed in \cite{AiCo1, AiCo2} with more elementary but less powerful methods.

Golan and Sapir undertook a systematic study of the subgroup $\vec{F}\leq F$  in \cite{GS, GS2}.
It should be mentioned that their work sits in a vast literature dedicated to the subgroups of $F$. In particular, Savchuk studied in \cite{Sav2, Sav} the stabilizers of  points in the unit interval along with their Schreier graphs. These  are maximal infinite index subgroups of $F$, and Savchuk asked whether they were the only such subgroups of $F$  \cite[Problem 2.5]{Sav}. A negative answer was provided by Golan and Sapir in \cite{GS2}, where Jones's oriented subgroup $\vec{F}$ was employed to exhibit a maximal subgroup of infinite index in $F$, whose elements do not have a common fixed point in $(0,1)$. 
The group $\vec{F}$ was also shown to be isomorphic to $F_3$ in \cite{GS} (see   \cite{Ren} for an alternative proof) and, 
hence,
an alternative description of the oriented subgroup was given \cite{Jo14, GS} as the stabilizer of a certain subset of dyadic rationals.
Thanks to this, it was proven that $\vec{F}$ coincides with its commensurator. By a classical result \cite[Section 3.4, Corollary 2]{Ma}, this implies that the quasi-regular representation associated with $\vec{F}$ is irreducible.

More recently, Jones \cite{Jo18} defined a new subgroup, this time of $F_3$: the so-called ternary oriented subgroup $\vec{F}_3$.
We briefly recall one motivation for introducing this subgroup.
As pointed out in \cite{Jo18}, the construction of knots and links 
from elements of $F$
of \cite{Jo14} can be   understood as follows. Firsty, we embed $F=F_2$ into $F_3$ by turning all the trivalent vertices of $(T_+,T_-)$
into
quadrivalent ones. Secondly, we join the two roots of the two trees and, thirdly, replace all the vertices by an appropriate crossing, to get a knot/link.
In \cite{Jo18} Jones extended this construction of knots/links to the whole Brown-Thompson group $F_3$ and, in order to describe  oriented knots, the group $\vec{F}_3$ was introduced.
It is 
therefore natural to perform an analysis of the ternary oriented subgroup, in the spirit of that done by Golan and Sapir for $\vec{F}=\vec{F}_2$.
This constitutes a substantial part of the present paper. As in the case of $\vec{F}=\vec{F}_2\leq F_2 = F$, the ternary oriented subgroup $\vec{F}_3$ gives rise to a non-parabolic maximal subgroup of infinite index in $F_3$. This  brings us to the question of finding maximal subgroups of infinite index that do not stabilise any point in $(0,1)$ in $F_k$ for all $k\geq 2$. 
For the moment there is no natural candidate for $\vec{F}_k$, with $k\geq 4$. Indeed, it is not known whether the Brown-Thompson groups $F_k, k\geq 4$ are good knot constructors, as is the case of $F_2$ and $F_3$. However, an algebraic description of $\vec F_2$ obtained by the machinery from 
 \cite{Ren} allows to define subgroups $H_k\leq F_k$ analogous to $\vec{F}_2$. It remains open whether these subgroups  give rise to maximal infinite index subgroups of $F_k$, as oriented $\vec{F}_2$ does in $F_2$.

We end this introduction by saying a few words on the structure of the paper and its main results.
In the first section we briefly recall some equivalent definitions and properties of the Thompson group $F$, of the Brown-Thompson groups $F_k$, $k\geq 2$, and of the oriented subgroups $\vec{F}_2$ and $\vec{F}_3$, while
the remaining sections are devoted to the main results of this paper. 
The first one is a description of  the ternary oriented subgroup as the stabilizer of a certain subset of the triadic rationals, see Section 2. 
As a corollary we get that $\vec{F}_3$ coincides with its commensurator, 
which again implies that the  quasi-regular representation of $F_3$ associated with $\vec{F}_3$ is 
  irreducible.
The second main result is that the ternary oriented subgroup is finitely generated, see Section 3.
The third main result is contained in  Section 4, where we exhibit 
a maximal infinite index subgroup of $F_3$ isomorphic with $\vec{F}_3$ and show that it does not stabilise any point in $(0,1)$.
This is in contrast with another  
family of maximal infinite index subgroups, the so-called parabolic subgroups. We mention that,
likewise $\vec{F}\leq F$, the ternary oriented subgroup is precisely contained in three subgroups of $F_3$.
In the last section we introduce a family of subgroups $H_k\leq F_k$, for all $k\geq 2$, which generalise the oriented Thompson group $\vec{F}$, note some of their properties and formulate the question of their maximality for $k\geq 3$.

\section{Preliminaries and notation}
We start this section recalling the definition and  some properties of the Thompson group $F=F_2$ and of the Brown-Thompson groups $F_k, k\geq 2$. For further information, we refer to \cite{CFP, B} and \cite{Brown} 
\footnote{We point out that  Brown
uses a slightly different notation: 
$F_k=F_{k,1}$.}.
The Thompson group $F=F_2$ can be described as the subgroup of the   orientation preserving
piecewise linear homeomorphisms of the unit interval $[0,1]$, which are 
differentiable everywhere but at finitely many dyadic rationals numbers and such that on the intervals of differentiability the derivatives are powers of $2$.  
The Thompson group has the following infinite presentation
$$
F_2=\langle x_0, x_1, \ldots \; | \; x_n x_l = x_l x_{n+1} \;\;  \forall \; l<n\rangle\; . 
$$
The elements of $F$ also admit a graphical representation, namely by  pairs of rooted planar binary trees $(T_+,T_-)$ with the same number of leaves.  As usual, we draw a pair of trees in the plane with one tree upside down on top of the other and join their   leaves as in Figures \ref{AAA} and \ref{genF}.
 
We call such pairs tree diagrams.
Two pairs of trees are equivalent if they differ by a pair of opposing carets, see   Figure \ref{AAA}. Thanks to this equivalence relation, the following rule defines the multiplication in $F$: $(T_+,T)\cdot (T,T_-):=(T_+,T_-)$. The trivial element is represented by any pair $(T,T)$  and the inverse of $(T_+,T_-)$ is just $(T_-,T_+)$. We call \emph{splits} the vertices in the top tree and \emph{merges} those in the bottom tree.
See Figure \ref{genF} for a graphical description of the generators of $F_2$.
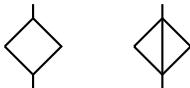
\begin{figure}
\caption{Pairs of opposing carets in $F=F_2$ and $F_3$.}\label{AAA}
\phantom{This text will be invisible} 
\[\begin{tikzpicture}[x=.75cm, y=.75cm,
    every edge/.style={
        draw,
      postaction={decorate,
                    decoration={markings}
                   }
        }
]

\draw[thick] (0,0)--(.5,.5)--(1,0)--(.5,-.5)--(0,0);
\draw[thick] (0.5,0.75)--(.5,.5);
\draw[thick] (0.5,-0.75)--(.5,-.5);
\node at (0,-1.2) {$\;$};
\end{tikzpicture}
\qquad 
\begin{tikzpicture}[x=.75cm, y=.75cm,
    every edge/.style={
        draw,
      postaction={decorate,
                    decoration={markings}
                   }
        }
]

\draw[thick] (0.5,0.75)--(.5,-.75);
\draw[thick] (0,0)--(.5,.5)--(1,0)--(.5,-.5)--(0,0);
\node at (0,-1.2) {$\;$};
\end{tikzpicture}
\]
\end{figure}
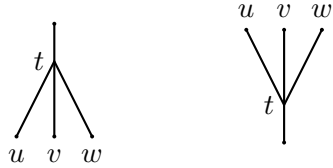
\begin{figure}
\caption{A split and a merge: $tu$ is a left edge, $tv$ a middle edge, $tw$ a right edge.}\label{fig-edge}
\phantom{This text will be invisible} 
 \[
\begin{tikzpicture}[x=1cm, y=1cm,
    every edge/.style={
        draw,
      postaction={decorate,
                    decoration={markings}
                   }
        }
]

\draw[thick] (0,0) -- (.5,1)--(1,0);
 \draw[thick] (.5,0)--(.5,1.5);

\fill (0,0)  circle[radius=.75pt];
\fill (.5,0)  circle[radius=.75pt];
\fill (1,0)  circle[radius=.75pt];
\fill (.5,1.5)  circle[radius=.75pt];

\node at (.3,1) {$\scalebox{1}{$t$}$};
\node at (0,-.25) {$\scalebox{1}{$u$}$};
\node at (.5,-.25) {$\scalebox{1}{$v$}$};
\node at (1,-.25) {$\scalebox{1}{$w$}$};

\end{tikzpicture}
\qquad\qquad
\begin{tikzpicture}[x=1cm, y=1cm,
    every edge/.style={
        draw,
      postaction={decorate,
                    decoration={markings}
                   }
        }
]

\draw[thick] (0,1.5) -- (.5,.5)--(1,1.5);
 \draw[thick] (.5,0)--(.5,1.5);

\fill (0,1.5)  circle[radius=.75pt];
\fill (.5,0)  circle[radius=.75pt];
\fill (1,1.5)  circle[radius=.75pt];
\fill (.5,1.5)  circle[radius=.75pt];

\node at (.3,.5) {$\scalebox{1}{$t$}$};
\node at (0,1.75) {$\scalebox{1}{$u$}$};
\node at (.5,1.75) {$\scalebox{1}{$v$}$};
\node at (1,1.75) {$\scalebox{1}{$w$}$};
\node at (1,-.25) {$\scalebox{1}{}$};

\end{tikzpicture}
\]
\end{figure}

\begin{figure}
\caption{The generators of $F=F_2$.}\label{genF}
\phantom{This text will be invisible} 
\[
\begin{tikzpicture}[x=.35cm, y=.35cm,
    every edge/.style={
        draw,
      postaction={decorate,
                    decoration={markings}
                   }
        }
]

\node at (-1.5,0) {$\scalebox{1}{$x_0=$}$};
\node at (-1.25,-3) {\;};

\draw[thick] (0,0) -- (2,2)--(4,0)--(2,-2)--(0,0);
 \draw[thick] (1,1) -- (2,0)--(3,-1);

 \draw[thick] (2,2)--(2,2.5);

 \draw[thick] (2,-2)--(2,-2.5);

\end{tikzpicture}
\;\;
\begin{tikzpicture}[x=.35cm, y=.35cm,
    every edge/.style={
        draw,
      postaction={decorate,
                    decoration={markings}
                   }
        }
]

\node at (-3.5,0) {$\scalebox{1}{$x_1=$}$};
\node at (-1.25,-3.25) {\;};

\draw[thick] (2,2)--(1,3)--(-2,0)--(1,-3)--(2,-2);

\draw[thick] (0,0) -- (2,2)--(4,0)--(2,-2)--(0,0);
 \draw[thick] (1,1) -- (2,0)--(3,-1);

 \draw[thick] (1,3)--(1,3.5);
 \draw[thick] (1,-3)--(1,-3.5);

\end{tikzpicture}
\;\;
\begin{tikzpicture}[x=.35cm, y=.35cm,
    every edge/.style={
        draw,
      postaction={decorate,
                    decoration={markings}
                   }
        }
]

\node at (-5.5,0) {$\scalebox{1}{$x_3=$}$};
\node at (6,0) {$\ldots$};

\draw[thick] (2,2)--(1,3)--(-2,0)--(1,-3)--(2,-2);
\draw[thick] (1,3)--(0,4)--(-4,0)--(0,-4)--(1,-3); 

\draw[thick] (0,0) -- (2,2)--(4,0)--(2,-2)--(0,0);
 \draw[thick] (1,1) -- (2,0)--(3,-1);

 \draw[thick] (0,4)--(0,4.5);
 \draw[thick] (0,-4)--(0,-4.5);

\end{tikzpicture}
\]
\end{figure}
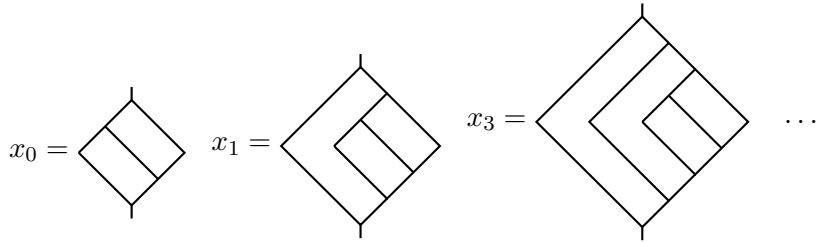
\begin{figure}
\caption{The generators of $F_3$.}\label{genF3}
\phantom{This text will be invisible} 
\[
\begin{tikzpicture}[x=.35cm, y=.35cm,
    every edge/.style={
        draw,
      postaction={decorate,
                    decoration={markings}
                   }
        }
]

\node at (-1.5,0) {$\scalebox{1}{$y_0=$}$};
\node at (-1.25,-3) {\;};

\draw[thick] (0,0) -- (2,2)--(4,0)--(2,-2)--(0,0);
\draw[thick] (1,1) -- (1,0)--(2,-2);
\draw[thick] (1,1) -- (2,0)--(3,-1);
\draw[thick] (2,2) -- (3,0)--(3,-1);

 \draw[thick] (2,2)--(2,2.5);
 \draw[thick] (2,-2)--(2,-2.5);

\end{tikzpicture}
\;\;
\begin{tikzpicture}[x=.35cm, y=.35cm,
    every edge/.style={
        draw,
      postaction={decorate,
                    decoration={markings}
                   }
        }
]

\node at (-1.5,0) {$\scalebox{1}{$y_1=$}$};
\node at (-1.25,-3.25) {\;};


\draw[thick] (0,0) -- (2,2)--(4,0)--(2,-2)--(0,0);
 \draw[thick] (1,0)--(2,-2);
\draw[thick] (3,0)--(2,1) -- (1,0); 
\draw[thick] (2,0)--(3,-1);
\draw[thick] (2,2) -- (2,0);
\draw[thick] (3,0)--(3,-1);


 \draw[thick] (2,2)--(2,2.5);
 \draw[thick] (2,-2)--(2,-2.5);
\end{tikzpicture}
\;\;
\begin{tikzpicture}[x=.35cm, y=.35cm,
    every edge/.style={
        draw,
      postaction={decorate,
                    decoration={markings}
                   }
        }
]

\node at (-3.5,0) {$\scalebox{1}{$y_2=$}$};
\node at (-1.25,-3.25) {\;};

\draw[thick] (2,2)--(1,3)--(-2,0)--(1,-3)--(2,-2);

\draw[thick] (0,0) -- (2,2)--(4,0)--(2,-2)--(0,0);
 \draw[thick] (1,1) -- (2,0)--(3,-1);

\draw[thick] (0,0) -- (2,2)--(4,0)--(2,-2)--(0,0);
\draw[thick] (1,1) -- (1,0)--(2,-2);
\draw[thick] (1,1) -- (2,0)--(3,-1);
\draw[thick] (2,2) -- (3,0)--(3,-1);

\draw[thick] (1,3) -- (-1,0)--(1,-3);

\node at (6,0) {$\ldots$};

 \draw[thick] (1,3)--(1,3.5);
 \draw[thick] (1,-3)--(1,-3.5);

\end{tikzpicture}
\]
\end{figure}
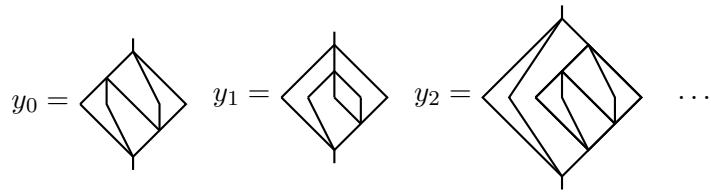

Similarly, for any $k\geq 2$, the Brown-Thompson group $F_k$ can be described as a group of   orientation preserving
piecewise linear homeomorphisms of the unit interval $[0,1]$, but in this case the points where the homeomorphisms may not be differentiable lie in $\IZ[1/k]$ and the slopes are powers of $k$. 
The Brown-Thompson group $F_k$ also admits an infinite presentation
$$
F_k=\langle y_0, y_1, \ldots \; | \; y_n y_l = y_l y_{n+k-1} \;\;  \forall \; l<n\rangle\; . 
$$
The Brown-Thompson group $F_k$ has the same description  
in terms of  trees as $F_2$, except that the trees are now $k$-ary: 
 the vertices are no longer $3$-valent, but $(k+1)$-valent. All the elements of $F_k$, $k\geq 2$, can be written in normal form,  as 
 $$
 y_{i_1}^{a_1}\cdots y_{i_n}^{a_n}y_{j_m}^{-b_m}\cdots y_{j_1}^{-b_1}
 $$ 
 with $i_1<\cdots <i_n\neq j_m>\cdots >j_1$. This representation is unique if one assumes that when both $y_i$ and $y_i^{-1}$ appear, then one among $y_j$ or $y_j^{-1}$, with $i<j<i+k$, appears as well.
 For a large part of this paper we will be mainly interested in $F_3$. In this case,
all the trees are subtrees of the tree of the standard triadic intervals, whose root is $[0,1]$ and its vertices are the standard triadic intervals (i.e., intervals of the form $[a3^{-l},(a+1)3^{-l}]$ for $a\in\IN$, $l\in\{0,1,\ldots, 3^l-1\}$), cf. \cite{B}.
See Figure \ref{genF3} for a graphical description of the generators of $F_3$. 

There is a natural embedding $\iota: F_2\to F_3$ obtained by replacing the $3$-valent vertices in a tree diagram of $F_2$ with $4$-valent vertices and joining the middle edges in the only possible planar way.
Under this embedding, we have $\iota(x_i)=y_{2i}$, see Figures
\ref{genF}
and 
\ref{genF3}. To ease the notation we will often omit the symbol of this map.

Let $(T_+,T_-)$ be any graphical representative of an element of $F_3$. 
We now briefly recall from \cite{Jo14, Jo18} how to obtain $\Gamma(T_+,T_-)$, which we call the {\bf 
  planar graph of $(T_+,T_-)$. }
We make a convention about how we draw tree diagrams. 
The leaves of the tree diagrams sit on the $x$-axis, precisely 
on the non-negative integers $\IN_0 :=\{0, 1, 2, \ldots \}$. 
The root  of the top tree is on the line $y=1$, while that of the 
bottom tree is on the line $y=-1$. 
{\it Any tree diagram $(T_+,T_-)$ partitions the strip bounded by the lines $y=1$ and $y=-1$ in regions that we colour in black and white, 
with the convention that the left-most region is black. 
Now the vertices of the 
planar graph  sit on the $x$-axis, precisely 
on 
$-1/2+2\IN_0 :=\{-1/2,1+1/2, 3+1/2, \ldots \}$ 
and there is precisely one vertex for  every black region. 
We draw an edge between two black regions whenever they meet at a $4$-valent vertex or at the root.}
For an element $(T_+,T_-)$ of $F_2$, we define the {\bf 
  planar graph of $(T_+,T_-)$} to be that of $\iota(T_+,T_-)\in F_3$.

Note that the 
planar graph of $(T_+,T_-)$ is
essentially the Tait graph of $\CL(T_+,T_-)$ 
except that we do not specify the signs of its edges. 

We are 
now in a position to define the oriented subgroups
\begin{align*}
\vec{F}_2&:=\{(T_+,T_-)\in F_2 \; | \; \Gamma(T_+,T_-) \textrm{ is $2$-colourable} \} \\ 
\vec{F}_3&:=\{(T_+,T_-)\in F_3\; | \; \Gamma(T_+,T_-) \textrm{ is $2$-colourable} \}
\end{align*}
where by being $2$-colourable we mean that it is possible to label the vertices of the graph with two colours such that whenever two vertices are connected by an edge, they have different colours.  
We denote by $+$ and $-$ the two colours.
Note that the definition does not depend on the specific representative $(T_+,T_-)$, \cite[
Section 4.1]{Jo14}. 
Since $\Gamma(T_+,T_-)$ is connected, if it is $2$-colourable, there are exactly two colourings. By convention we choose the one in which the left-most vertex has colour $+$. 
We observe that if we cut the 
 planar graph of $(T_+,T_-)$ along the $x$-axis we get two subgraphs: one in the upper-half plane which we denote by $\Gamma_+\equiv \Gamma_+(T_+)$, one in the lower-half plane which we denote by $\Gamma_-\equiv\Gamma_-(T_-)$. Since the number of vertices of $\Gamma_\pm(T_\pm)$ is equal to the number of $4$-valent vertices of $T_\pm$ plus $1$, while the number of edges of $\Gamma_\pm(T_\pm)$ is equal to the number of $4$-valent vertices of $T_\pm$, it follows that the graphs  $\Gamma_+(T_+)$ and  $\Gamma_-(T_-)$ are trees. 
The trees $\Gamma_+$ and $\Gamma_-$ are always $2$-colourable, so $\Gamma(T_+,T_-)$ is $2$-colourable precisely when the colourings of  $\Gamma_+$ and $\Gamma_-$  are the same.
We observe that for 
$\Gamma(T_+,T_-)$ being $2$-colourable is the same  as being bipartite.

\begin{remark}\label{rem1} 
It is worth mentioning that if we restrict the map $\iota$  to $\vec{F}_2$, its image is contained in $\vec{F}_3$. Therefore, the elements $y_{2i}y_{2i+2}=\iota(x_ix_{i+1})\in \vec{F}_3$ since $x_ix_{i+1}\in \vec{F}_2$, \cite{GS}.
As we will see in Example \ref{esey2}, the restriction of $\iota$ to $\vec{F}_2$ is not surjective to $\vec{F}_3$. 
\end{remark}
It is appropriate to say a word of warning: there exists another group denoted by $\vec{F}_3$ in the literature  \cite{GS}, which is a subgroup of $F_2$.

 \section{The ternary oriented Thompson group as a stabilizer subgroup}\label{sec-2}
We want to describe the ternary oriented Thompson group as the stabilizer (under the natural action) of a certain subset of triadic rationals. This result should be compared with  \cite[Theorem 2]{GS}, where $\vec{F}=\vec{F}_2$ is realized as the stabilizer of the set of dyadic fractions in the unit interval $[0,1]$ with odd sums of digits. 

Consider a tree in the upper half-plane and its leaves on the $x$-axis as usual.
To each vertex 
 $v$ of a tree we associate a natural number $c(v)$ which we call its {\bf weight}, as follows. 
{\it Given a vertex, there exists a unique minimal path from the root of the tree to the vertex. This path is made by a collection of left, middle, right edges, and may be represented by a word $w_1 1w_2 1\cdots 1 w_n$ in the letters
$\{0, 1, 2\}$ ($0$ stands for a left edge, $1$ for a middle edge, $2$ for a right edge), where $w_1, \ldots , w_{n-1}$ are words that do not contain the letter $1$, $w_n$ can have $1$ only as its last letter.
We call $\{w_{2k+1}\}_{k\geq 0}$ the odd words and $\{w_{2k}\}_{k\geq 0}$ the even words.
The weight of $v$ 
is the sum of the number of digits equal to $1$, plus the number of digits equal to $2$ in the odd words, plus the number of digits equal to $0$ in the even words.} 
When we compute the weight of a leaf in a tree diagram, 
sometimes we use the symbol $c_+$ or $c_-$ to distinguish which tree we are considering ($c_+$ for the top tree, $c_-$ for the reflected bottom tree).
{\it Similarly, we define the number $d(v)$ (and $d_\pm(v)$ if   we want to specify the tree) which counts the number of middle edges met in the path from the root to $v$.} 
Here follow a couple of easy lemmas that will come in handy in the sequel. The proofs are   similar and they can all be done by induction on the length of the path. 
We provide a proof only of the second one.
\begin{lemma}\label{lemma1}
If the following configurations occur in $T_+$
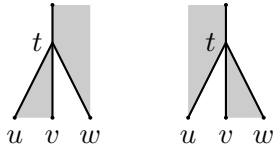
 \[
\begin{tikzpicture}[x=1cm, y=1cm,
    every edge/.style={
        draw,
      postaction={decorate,
                    decoration={markings}
                   }
        }
]

\draw[thick] (0,0) -- (.5,1)--(1,0);
 \draw[thick] (.5,0)--(.5,1.5);

\fill [color=black,opacity=0.2]   (0,0)--(0.5,1)--(.5,0); 
\fill [color=black,opacity=0.2]   (.5,1)--(.5,1.5)--(1,1.5)--(1,0)--(.5,1);

\fill (0,0)  circle[radius=.75pt];
\fill (.5,0)  circle[radius=.75pt];
\fill (1,0)  circle[radius=.75pt];
\fill (.5,1.5)  circle[radius=.75pt];

\node at (0,-.25) {$\scalebox{1}{$u$}$};
\node at (.5,-.25) {$\scalebox{1}{$v$}$};
\node at (1,-.25) {$\scalebox{1}{$w$}$};
\node at (.3,1) {$\scalebox{1}{$t$}$};

\end{tikzpicture}
\qquad 
\begin{tikzpicture}[x=1cm, y=1cm,
    every edge/.style={
        draw,
      postaction={decorate,
                    decoration={markings}
                   }
        }
]

\draw[thick] (0,0) -- (.5,1)--(1,0);
 \draw[thick] (.5,0)--(.5,1.5);

\fill [color=black,opacity=0.2]   (1,0)--(0.5,1)--(.5,0); 
\fill [color=black,opacity=0.2]   (.5,1)--(.5,1.5)--(0,1.5)--(0,0)--(.5,1);

\fill (0,0)  circle[radius=.75pt];
\fill (.5,0)  circle[radius=.75pt];
\fill (1,0)  circle[radius=.75pt];
\fill (.5,1.5)  circle[radius=.75pt];

\node at (0,-.25) {$\scalebox{1}{$u$}$};
\node at (.5,-.25) {$\scalebox{1}{$v$}$};
\node at (1,-.25) {$\scalebox{1}{$w$}$};
\node at (.3,1) {$\scalebox{1}{$t$}$};

\end{tikzpicture}
\]
then  in the first case $d_+(t), d_+(u), d_+(w)\in 2\IN_0+1$ and $d_+(v)\in 2\IN_0$, while in the second  
 $d_+(v)\in 2\IN_0+1$ and $d_+(t), d_+(u), d_+(w)\in 2\IN_0$. 
\end{lemma}
Recall from Section 1 that the vertices of tree diagrams and the planar graphs 
$\Gamma(T_+,T_-)$
 can be identified with the non-negative integers and $-1/2+2\IN_0 :=\{-1/2,1+1/2, 3+1/2, \ldots \}$, respectively.
 \begin{lemma}\label{lemma-sign} 
Given an element $(T_+,T_-)\in F_3$, we have $c_\pm(2i)=c_\pm(2i-1)$ for all $i\geq 1$.
Moreover,  the colouring of the vertex $2i-1/2$ in $\Gamma_\pm$ 
 is $+$ if $c_\pm(2i)$ is even and it is $-$ if  $c_\pm(2i)$  is odd, $i\geq 0$. 
\end{lemma}
\begin{proof}
The proof is done by induction on the length of the path. The case of a path of length $1$ is obvious.
Suppose that the claim is true for all the paths of length $n$ and consider a path of length $n+1$. 
The $n$-th edge of this path  can either be a left, middle, or right edge. 
 Depending on this 
 and the 
 colours of the regions at distance at most $2$ from the last edge (there are two possible cases) we are led to consider
 six cases (depicted below). 
In the first two cases a split is attached below a right edge.
In the third and fourth cases a split is attached below a middle edge.  In the last two cases a split is attached below a left edge. Except for the leaves corresponding to the new split, the colourings (and the weight) of the other vertices are determined by induction. We write next to each vertex 
  its corresponding weight and see that 
  its parity matches 
   the signs in the $\Gamma_\pm$-graph (in all these cases we will repeatedly use Lemma \ref{lemma1} in order to determine how to compute the weight).
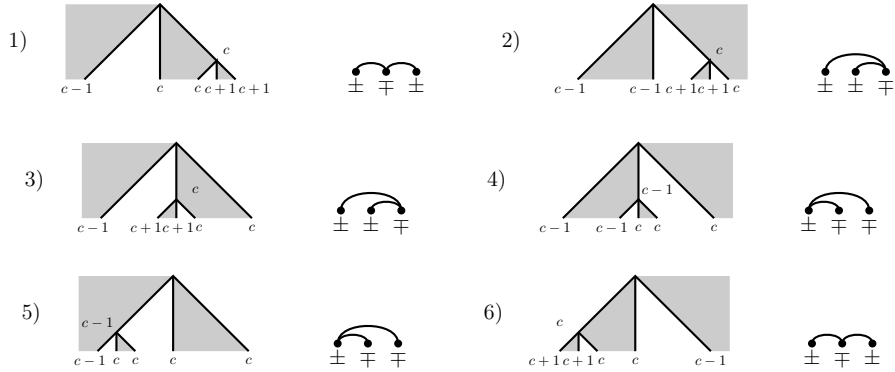
\[
\begin{tikzpicture}[x=.5cm, y=.5cm,
    every edge/.style={
        draw,
      postaction={decorate,
                    decoration={markings}
                   }
        }
]

\node at (-.75,1) {$\scalebox{.75}{1)}$};

\draw[thick] (1,0) -- (3,2)--(5,0);
 \draw[thick] (3,0) -- (3,2);
\draw[thick] (4,0) -- (4.5,.5)--(4.5,0);

\fill [color=black,opacity=0.2]   (.5,0)--(.5,2)--(3,2)--(1,0); 
 \fill [color=black,opacity=0.2]   (4.5,0)--(4.5,.5)--(5,0); 
\fill [color=black,opacity=0.2]   (3,0)--(3,2)--(4.5,0.5)--(4,0);

  \node at (4,-.25) {$\scalebox{.5}{$c$}$};
  \node at (.8,-.25) {$\scalebox{.5}{$c-1$}$};
 \node at (3,-.25) {$\scalebox{.5}{$c$}$};
 \node at (4.75,.75) {$\scalebox{.5}{$c$}$};
 \node at (5.5,-.25) {$\scalebox{.5}{$c+1$}$};
 \node at (4.6,-.25) {$\scalebox{.5}{$c+1$}$};

 \node at (5.5,-.25) {\phantom{$\scalebox{.5}{$c+1$}$}};

\end{tikzpicture}
\qquad
\begin{tikzpicture}[x=.4cm, y=.4cm,
    every edge/.style={
        draw,
      postaction={decorate,
                    decoration={markings}
                   }
        }
]

\draw[thick] (1,0) to[out=90,in=90] (2,0);
\draw[thick] (2,0) to[out=90,in=90] (3,0);

\fill (1,0)  circle[radius=1.5pt];
\fill (2,0)  circle[radius=1.5pt];
\fill (3,0)  circle[radius=1.5pt];
 
 \node at (1,-.5) {$\scalebox{.75}{$\pm$}$};
\node at (2,-.5) {$\scalebox{.75}{$\mp$}$};
\node at (3,-.5) {$\scalebox{.75}{$\pm$}$};

\end{tikzpicture}
\qquad
\begin{tikzpicture}[x=.5cm, y=.5cm,
    every edge/.style={
        draw,
      postaction={decorate,
                    decoration={markings}
                   }
        }
]

\node at (-.75,1) {$\scalebox{.75}{2)}$};

\draw[thick] (1,0) -- (3,2)--(5,0);
 \draw[thick] (3,0) -- (3,2);
\draw[thick] (4,0) -- (4.5,.5)--(4.5,0);

 \fill [color=black,opacity=0.2]   (4,0)--(4.5,.5)--(4.5,0)--(4,0); 
\fill [color=black,opacity=0.2]   (1,0)--(3,2)--(3,0)--(2.5,0)--(2,0); 
\fill [color=black,opacity=0.2]   (3,2)--(5.5,2)--(5.5,0)--(5,0);

  \node at (3.75,-.25) {$\scalebox{.5}{$c+1$}$};
  \node at (.8,-.25) {$\scalebox{.5}{$c-1$}$};
 \node at (2.8,-.25) {$\scalebox{.5}{$c-1$}$};
 \node at (4.75,.75) {$\scalebox{.5}{$c$}$};
 \node at (5.2,-.25) {$\scalebox{.5}{$c$}$};
 \node at (4.6,-.25) {$\scalebox{.5}{$c+1$}$};

\end{tikzpicture}
\qquad
\begin{tikzpicture}[x=.4cm, y=.4cm,
    every edge/.style={
        draw,
      postaction={decorate,
                    decoration={markings}
                   }
        }
]

\draw[thick] (1,0) to[out=90,in=90] (3,0);
\draw[thick] (2,0) to[out=90,in=90] (3,0);

\fill (1,0)  circle[radius=1.5pt];
\fill (2,0)  circle[radius=1.5pt];
\fill (3,0)  circle[radius=1.5pt];
 
 \node at (1,-.5) {$\scalebox{.75}{$\pm$}$};
\node at (2,-.5) {$\scalebox{.75}{$\pm$}$};
\node at (3,-.5) {$\scalebox{.75}{$\mp$}$};

\end{tikzpicture}
\]
\[
\begin{tikzpicture}[x=.5cm, y=.5cm,
    every edge/.style={
        draw,
      postaction={decorate,
                    decoration={markings}
                   }
        }
]

\node at (-.75,1) {$\scalebox{.75}{3)}$};

\draw[thick] (1,0) -- (3,2)--(5,0);
 \draw[thick] (3,0) -- (3,2);
\draw[thick] (2.5,0) -- (3,.5)--(3.5,0);

\fill [color=black,opacity=0.2]   (.5,0)--(.5,2)--(3,2)--(1,0); 
\fill [color=black,opacity=0.2]   (2.5,0)--(3,.5)--(3,0); 
\fill [color=black,opacity=0.2]   (3,2)--(5,0)--(3.5,0)--(3,0.5)--(3,2);

  \node at (3.5,.75) {$\scalebox{.5}{$c$}$};
 \node at (2.2,-.25) {$\scalebox{.5}{$c+1$}$};
 \node at (.8,-.25) {$\scalebox{.5}{$c-1$}$};
 \node at (3.05,-.25) {$\scalebox{.5}{$c+1$}$};
 \node at (3.6,-.25) {$\scalebox{.5}{$c$}$};
 \node at (5,-.25) {$\scalebox{.5}{$c$}$};

\end{tikzpicture}
\qquad
\begin{tikzpicture}[x=.4cm, y=.4cm,
    every edge/.style={
        draw,
      postaction={decorate,
                    decoration={markings}
                   }
        }
]

\draw[thick] (1,0) to[out=90,in=90] (3,0);
\draw[thick] (2,0) to[out=90,in=90] (3,0);

\fill (1,0)  circle[radius=1.5pt];
\fill (2,0)  circle[radius=1.5pt];
\fill (3,0)  circle[radius=1.5pt];
 
 \node at (1,-.5) {$\scalebox{.75}{$\pm$}$};
\node at (2,-.5) {$\scalebox{.75}{$\pm$}$};
\node at (3,-.5) {$\scalebox{.75}{$\mp$}$};

\end{tikzpicture}
\qquad
\begin{tikzpicture}[x=.5cm, y=.5cm,
    every edge/.style={
        draw,
      postaction={decorate,
                    decoration={markings}
                   }
        }
]

\node at (-.75,1) {$\scalebox{.75}{4)}$};

\draw[thick] (1,0) -- (3,2)--(5,0); 
\draw[thick] (3,0) -- (3,2);
\draw[thick] (2.5,0) -- (3,.5)--(3.5,0);

 \fill [color=black,opacity=0.2]   (3.5,0)--(3,.5)--(3,0)--(3.5,0); 
\fill [color=black,opacity=0.2]   (1,0)--(1.5,.5)--(3,2)--(3,.5)--(2.5,0)--(2,0); 
\fill [color=black,opacity=0.2]   (3,2)--(5.5,2)--(5.5,0)--(5,0);

  \node at (3.5,.75) {$\scalebox{.5}{$c-1$}$};
 \node at (2.3,-.25) {$\scalebox{.5}{$c-1$}$};
 \node at (.8,-.25) {$\scalebox{.5}{$c-1$}$};
 \node at (3,-.25) {$\scalebox{.5}{$c$}$};
 \node at (3.5,-.25) {$\scalebox{.5}{$c$}$};
 \node at (5,-.25) {$\scalebox{.5}{$c$}$};

\end{tikzpicture}
\qquad
\begin{tikzpicture}[x=.4cm, y=.4cm,
    every edge/.style={
        draw,
      postaction={decorate,
                    decoration={markings}
                   }
        }
]

\draw[thick] (1,0) to[out=90,in=90] (3,0);
\draw[thick] (1,0) to[out=90,in=90] (2,0);

\fill (1,0)  circle[radius=1.5pt];
\fill (2,0)  circle[radius=1.5pt];
\fill (3,0)  circle[radius=1.5pt];
 
 \node at (1,-.5) {$\scalebox{.75}{$\pm$}$};
\node at (2,-.5) {$\scalebox{.75}{$\mp$}$};
\node at (3,-.5) {$\scalebox{.75}{$\mp$}$};

\end{tikzpicture}
\]
\[
\begin{tikzpicture}[x=.5cm, y=.5cm,
    every edge/.style={
        draw,
      postaction={decorate,
                    decoration={markings}
                   }
        }
]

\node at (-.75,1) {$\scalebox{.75}{5)}$};

\draw[thick] (1,0) -- (3,2)--(5,0);
\draw[thick] (1.5,.5) -- (1.5,0);
\draw[thick] (1.5,.5) -- (2,0);
\draw[thick] (3,0) -- (3,2);
 
\fill [color=black,opacity=0.2]   (.5,0)--(.5,2)--(3,2)--(1,0); 
\fill [color=black,opacity=0.2]   (1.5,0)--(1.5,.5)--(2,0); 
\fill [color=black,opacity=0.2]   (3,2)--(5,0)--(3,0)--(3,0.5)--(3,2); 

 \node at (1,.75) {$\scalebox{.5}{$c-1$}$};
 \node at (.8,-.25) {$\scalebox{.5}{$c-1$}$};
 \node at (1.5,-.25) {$\scalebox{.5}{$c$}$};
 \node at (2,-.25) {$\scalebox{.5}{$c$}$};
 \node at (3,-.25) {$\scalebox{.5}{$c$}$};
 \node at (5,-.25) {$\scalebox{.5}{$c$}$};

\end{tikzpicture}
\qquad
\begin{tikzpicture}[x=.4cm, y=.4cm,
    every edge/.style={
        draw,
      postaction={decorate,
                    decoration={markings}
                   }
        }
]

\draw[thick] (1,0) to[out=90,in=90] (3,0);
\draw[thick] (1,0) to[out=90,in=90] (2,0);

\fill (1,0)  circle[radius=1.5pt];
\fill (2,0)  circle[radius=1.5pt];
\fill (3,0)  circle[radius=1.5pt];
 
 \node at (1,-.5) {$\scalebox{.75}{$\pm$}$};
\node at (2,-.5) {$\scalebox{.75}{$\mp$}$};
\node at (3,-.5) {$\scalebox{.75}{$\mp$}$};

\end{tikzpicture}
\qquad
\begin{tikzpicture}[x=.5cm, y=.5cm,
    every edge/.style={
        draw,
      postaction={decorate,
                    decoration={markings}
                   }
        }
]

\node at (-.75,1) {$\scalebox{.75}{6)}$};

\draw[thick] (1,0) -- (3,2)--(5,0);
\draw[thick] (1.5,.5) -- (1.5,0);
\draw[thick] (1.5,.5) -- (2,0);
\draw[thick] (3,0) -- (3,2);
 
\fill [color=black,opacity=0.2]   (1,0)--(1.5,.5)--(1.5,0)--(1,0); 
 \fill [color=black,opacity=0.2]   (2,0)--(1.5,.5)--(3,2)--(3,0)--(2.5,0)--(2,0); 
\fill [color=black,opacity=0.2]   (3,2)--(5.5,2)--(5.5,0)--(5,0);

 \node at (1,.75) {$\scalebox{.5}{$c$}$};
 \node at (.6,-.25) {$\scalebox{.5}{$c+1$}$};
 \node at (1.5,-.25) {$\scalebox{.5}{$c+1$}$};
 \node at (2.1,-.25) {$\scalebox{.5}{$c$}$};
 \node at (3,-.25) {$\scalebox{.5}{$c$}$};
 \node at (5,-.25) {$\scalebox{.5}{$c-1$}$};

\end{tikzpicture}
\qquad
\begin{tikzpicture}[x=.4cm, y=.4cm,
    every edge/.style={
        draw,
      postaction={decorate,
                    decoration={markings}
                   }
        }
]

\draw[thick] (1,0) to[out=90,in=90] (2,0);
\draw[thick] (2,0) to[out=90,in=90] (3,0);

\fill (1,0)  circle[radius=1.5pt];
\fill (2,0)  circle[radius=1.5pt];
\fill (3,0)  circle[radius=1.5pt];
 
 \node at (1,-.5) {$\scalebox{.75}{$\pm$}$};
\node at (2,-.5) {$\scalebox{.75}{$\mp$}$};
\node at (3,-.5) {$\scalebox{.75}{$\pm$}$};

\end{tikzpicture}
\]
\end{proof}
As an immediate consequence of the previous lemma we get the following description of the ternary oriented Thompson group.
\begin{proposition}\label{prop1}
It holds
$$
\vec{F}_3 =\{(T_+,T_-)\in F_3\; | \; c_+(i)\equiv_2 c_-(i) \; \forall \;  i\geq 0\}\; .
$$
\end{proposition}
The natural action of an element $\phi\in F_3$ on the numbers in $[0,1]$, expressed in ternary expansion is described in Figure \ref{compute}.  
The number $t$ enters into the top of the tree diagram, 
follows a path towards the root of the bottom tree according to the displayed rules and what emerges at the bottom is the image of $t$ under the homeomorphism $\phi$ 
,  cf. \cite{BM}. Note that there is a change of direction only when the number comes across a vertex of degree $4$ (that is, the number is unchanged when it comes across a leaf). 
Now follow some examples with explicit computations. 
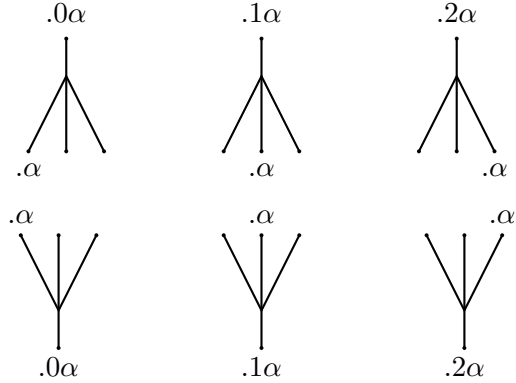
\begin{figure}
\caption{The local rules for computing the action of $F_3$ on numbers expressed in ternary expansion.}\label{compute}
\phantom{This text will be invisible} 
 \[
 \begin{tikzpicture}[x=1cm, y=1cm,
    every edge/.style={
        draw,
      postaction={decorate,
                    decoration={markings}
                   }
        }
]

\draw[thick] (0,0) -- (.5,1)--(1,0);
 \draw[thick] (.5,0)--(.5,1.5);

\fill (0,0)  circle[radius=.75pt];
\fill (.5,0)  circle[radius=.75pt];
\fill (1,0)  circle[radius=.75pt];
\fill (.5,1.5)  circle[radius=.75pt];

\node at (.5,1.85) {$\scalebox{1}{$.0\alpha$}$};
\node at (0,-.25) {$\scalebox{1}{$.\alpha$}$};

\end{tikzpicture}
\qquad\qquad
\begin{tikzpicture}[x=1cm, y=1cm,
    every edge/.style={
        draw,
      postaction={decorate,
                    decoration={markings}
                   }
        }
]

\draw[thick] (0,0) -- (.5,1)--(1,0);
 \draw[thick] (.5,0)--(.5,1.5);

\fill (0,0)  circle[radius=.75pt];
\fill (.5,0)  circle[radius=.75pt];
\fill (1,0)  circle[radius=.75pt];
\fill (.5,1.5)  circle[radius=.75pt];

\node at (.5,1.85) {$\scalebox{1}{$.1\alpha$}$};
\node at (.5,-.25) {$\scalebox{1}{$.\alpha$}$};

\end{tikzpicture}
\qquad\qquad
\begin{tikzpicture}[x=1cm, y=1cm,
    every edge/.style={
        draw,
      postaction={decorate,
                    decoration={markings}
                   }
        }
]

\draw[thick] (0,0) -- (.5,1)--(1,0);
 \draw[thick] (.5,0)--(.5,1.5);

\fill (0,0)  circle[radius=.75pt];
\fill (.5,0)  circle[radius=.75pt];
\fill (1,0)  circle[radius=.75pt];
\fill (.5,1.5)  circle[radius=.75pt];

\node at (.5,1.85) {$\scalebox{1}{$.2\alpha$}$};
\node at (1,-.25) {$\scalebox{1}{$.\alpha$}$};

\end{tikzpicture}
 \]
 \[
\begin{tikzpicture}[x=1cm, y=1cm,
    every edge/.style={
        draw,
      postaction={decorate,
                    decoration={markings}
                   }
        }
]

\draw[thick] (0,1.5) -- (.5,.5)--(1,1.5);
 \draw[thick] (.5,0)--(.5,1.5);

\fill (0,1.5)  circle[radius=.75pt];
\fill (.5,0)  circle[radius=.75pt];
\fill (1,1.5)  circle[radius=.75pt];
\fill (.5,1.5)  circle[radius=.75pt];

\node at (.5,-.25) {$\scalebox{1}{$.0\alpha$}$};
\node at (0,1.75) {$\scalebox{1}{$.\alpha$}$};
\node at (1,-.25) {$\scalebox{1}{}$};

\end{tikzpicture}
\qquad\qquad
\begin{tikzpicture}[x=1cm, y=1cm,
    every edge/.style={
        draw,
      postaction={decorate,
                    decoration={markings}
                   }
        }
]

\draw[thick] (0,1.5) -- (.5,.5)--(1,1.5);
 \draw[thick] (.5,0)--(.5,1.5);

\fill (0,1.5)  circle[radius=.75pt];
\fill (.5,0)  circle[radius=.75pt];
\fill (1,1.5)  circle[radius=.75pt];
\fill (.5,1.5)  circle[radius=.75pt];

\node at (.5,-.25) {$\scalebox{1}{$.1\alpha$}$};
\node at (.5,1.75) {$\scalebox{1}{$.\alpha$}$};
\node at (1,-.25) {$\scalebox{1}{}$};

\end{tikzpicture}
\qquad\qquad
\begin{tikzpicture}[x=1cm, y=1cm,
    every edge/.style={
        draw,
      postaction={decorate,
                    decoration={markings}
                   }
        }
]

\draw[thick] (0,1.5) -- (.5,.5)--(1,1.5);
 \draw[thick] (.5,0)--(.5,1.5);

\fill (0,1.5)  circle[radius=.75pt];
\fill (.5,0)  circle[radius=.75pt];
\fill (1,1.5)  circle[radius=.75pt];
\fill (.5,1.5)  circle[radius=.75pt];

\node at (.5,-.25) {$\scalebox{1}{$.2\alpha$}$};
\node at (1,1.75) {$\scalebox{1}{$.\alpha$}$};
\node at (1,-.25) {$\scalebox{1}{}$};

\end{tikzpicture}
\]
\end{figure}

\begin{example}\label{exe2}
 The generator $y_0$ is an example of an element in $F_3\setminus \vec{F}_3$. Below are the tree diagram, the shaded strip,     the corresponding $\Gamma$-graph which is not $2$-colourable/bipartite and the description of $y_0$ in terms of the action on $[0,1]$.
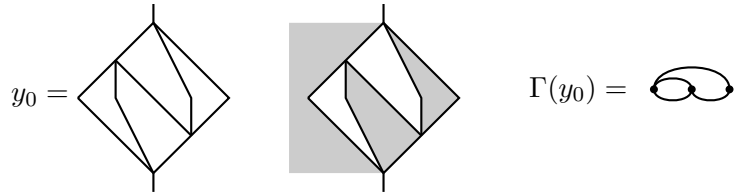
\[
\begin{tikzpicture}[x=.5cm, y=.5cm,
    every edge/.style={
        draw,
      postaction={decorate,
                    decoration={markings}
                   }
        }
]

\node at (-1,0) {$\scalebox{1}{$y_0=$}$};

\draw[thick] (0,0) -- (2,2)--(4,0)--(2,-2)--(0,0);
\draw[thick] (1,1) -- (1,0)--(2,-2);
\draw[thick] (1,1) -- (2,0)--(3,-1);
\draw[thick] (2,2) -- (3,0)--(3,-1);

 \draw[thick] (2,2)--(2,2.5);
 \draw[thick] (2,-2)--(2,-2.5);

\end{tikzpicture}
\qquad
\begin{tikzpicture}[x=.5cm, y=.5cm,
    every edge/.style={
        draw,
      postaction={decorate,
                    decoration={markings}
                   }
        }
]

\draw[thick] (0,0) -- (2,2)--(4,0)--(2,-2)--(0,0);
\draw[thick] (1,1) -- (1,0)--(2,-2);
\draw[thick] (1,1) -- (2,0)--(3,-1);
\draw[thick] (2,2) -- (3,0)--(3,-1);

\fill [color=black,opacity=0.2]   (2,-2)--(-.5,-2)--(-.5,2)--(2,2)--(2,2)--(0,0)--(2,-2); 
\fill [color=black,opacity=0.2]   (1,1)--(3,-1)--(2,-2)--(1,0)--(1,1); 
\fill [color=black,opacity=0.2]   (2,2)--(4,0)--(3,-1)--(3,0)--(2,2);

 \draw[thick] (2,2)--(2,2.5);
 \draw[thick] (2,-2)--(2,-2.5);

\end{tikzpicture}
\qquad 
\begin{tikzpicture}[x=.5cm, y=.5cm,
    every edge/.style={
        draw,
      postaction={decorate,
                    decoration={markings}
                   }
        }
]

\draw[thick] (1,0) to[out=90,in=90] (3,0);
\draw[thick] (1,0) to[out=90,in=90] (2,0);

\draw[thick] (1,0) to[out=-90,in=-90] (2,0); 
 \draw[thick] (2,0) to[out=-90,in=-90] (3,0);

\fill (1,0)  circle[radius=1.5pt];
\fill (2,0)  circle[radius=1.5pt];
\fill (3,0)  circle[radius=1.5pt];
 
 \node at (1,-2.5) {};
 \node at (-1,0) {$\Gamma(y_0)=$};

\end{tikzpicture}
\]
\begin{align*}
y_0(t)=
  \left\{ {\begin{array}{ll}
   .0\alpha & \textrm{ if $t=.00\alpha$} \\
   .1\alpha & \textrm{ if $t=.01\alpha$} \\
   .20\alpha & \textrm{ if $t=.02\alpha$} \\
   .21\alpha & \textrm{ if $t=.1\alpha$} \\
   .22\alpha & \textrm{ if $t=.2\alpha$} \\
  \end{array} } \right.
\end{align*}
One can also show that this element does not belong to $\vec{F}_3$ by computing the weights. Indeed, one has $c_+(0)=0=c_-(0)$, $c_+(1)=1=c_-(1)$, $c_+(2)=1=c_-(2)$, $c_+(3)=1\neq 2=c_-(3)$, $c_+(4)=1\neq 2=c_-(4)$. 
\end{example}

\begin{example}\label{esey2}
Thanks to the following element we show that the inclusion $\vec{F}_2\leq \vec{F}_3$ is actually strict. 
 Below are the tree diagram, the shaded strip, the corresponding $\Gamma$-graph 
  and the description of the element 
 in terms of the action on $[0,1]$
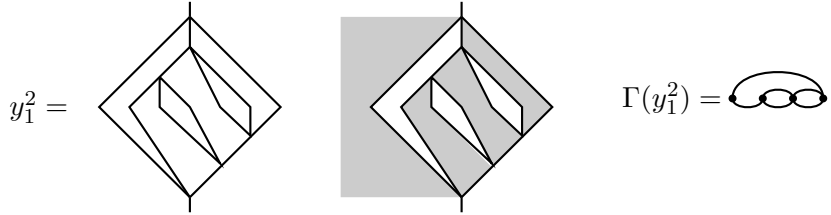
\[
\begin{tikzpicture}[x=.4cm, y=.4cm,
    every edge/.style={
        draw,
      postaction={decorate,
                    decoration={markings}
                   }
        }
]

\node at (0,0) {$\scalebox{1}{$y_1^2=$}$};

\draw[thick] 
(5,3)--(8,0)--(5,-3);
\draw[thick] (5,-3)--(2,0) -- (5,3);
\draw[thick] (5,-3)--(3,0) -- (5,2)--(5,3);
\draw[thick] (5,2)--(7,0)--(7,-.95)--(6,0)--(5,2);
\draw[thick] (4,1)--(4,0) -- (6,-1.875)--(5,0)--(4,1);

 \draw[thick] (5,3)--(5,3.5);
 \draw[thick] (5,-3)--(5,-3.5);

\end{tikzpicture}
\qquad 
\begin{tikzpicture}[x=.4cm, y=.4cm,
    every edge/.style={
        draw,
      postaction={decorate,
                    decoration={markings}
                   }
        }
]

 \draw[thick] (5,3)--(5,3.5);
 \draw[thick] (5,-3)--(5,-3.5);
 
\draw[thick] 
(5,3)--(8,0)--(5,-3);
\draw[thick] (5,-3)--(2,0) -- (5,3);
\draw[thick] (5,-3)--(3,0) -- (5,2)--(5,3);
\draw[thick] (5,2)--(7,0)--(7,-.95)--(6,0)--(5,2);
\draw[thick] (4,1)--(4,0) -- (6,-1.875)--(5,0)--(4,1);
\fill [color=black,opacity=0.2]   (1,0)--(1,3)--(5,3)--(2,0)--(5,-3)--(1,-3); 
\fill [color=black,opacity=0.2]   (5,3)--(8,0)--(7,-1)--(7,0)--(5,2)--(5,3); 
\fill [color=black,opacity=0.2]   (5,2)--(6,0)--(7,-1)--(6,-2)--(5,0)--(4,1)--(5,2); 
\fill [color=black,opacity=0.2]   (4,1)--(4,0)--(6,-2)--(5,-3)--(3,0)--(4,1); 


\end{tikzpicture}
\qquad 
\begin{tikzpicture}[x=.4cm, y=.4cm,
    every edge/.style={
        draw,
      postaction={decorate,
                    decoration={markings}
                   }
        }
]

 \draw[thick] (2,0) to[out=90,in=90] (5,0); 
\draw[thick] (4,0) to[out=90,in=90] (5,0); 
\draw[thick] (3,0) to[out=90,in=90] (4,0);

  \draw[thick] (2,0) to[out=-90,in=-90] (3,0);  
 \draw[thick] (3,0) to[out=-90,in=-90] (4,0);  
 \draw[thick] (4,0) to[out=-90,in=-90] (5,0);

 \fill (2,0)  circle[radius=1.5pt];
\fill (3,0)  circle[radius=1.5pt];
\fill (4,0)  circle[radius=1.5pt];
\fill (5,0)  circle[radius=1.5pt];
 
 \node at (1,-3.5) {};
 \node at (0,0) {$\Gamma(y_1^2)=$};

\end{tikzpicture}
\]
\begin{align*}
y_1^2(t)=
  \left\{ {\begin{array}{ll}
   .0\alpha & \textrm{ if $t=.0\alpha$} \\
   .1\alpha & \textrm{ if $t=.100\alpha$} \\
   .20\alpha & \textrm{ if $t=.101\alpha$} \\
   .21\alpha & \textrm{ if $t=.102\alpha$} \\
   .220\alpha & \textrm{ if $t=.11\alpha$} \\
   .221\alpha & \textrm{ if $t=.12\alpha$} \\
   .222\alpha & \textrm{ if $t=.2\alpha$} \\
  \end{array} } \right.
\end{align*}
One can also check that the weights of the leaves are the same mod $2$. Indeed, $c_+(0)=0=c_-(0)$, $c_+(1)=3\equiv_2 1=c_-(1)$, $c_+(2)=3\equiv_2 1=c_-(2)$, $c_+(3)=2=c_-(3)$, $c_+(4)=2=c_-(4)$, $c_+(5)=1\equiv_2 3=c_-(5)$, $c_+(6)=1\equiv_2 3=c_-(6)$. 
 \end{example}
\begin{example}\label{ese3}
The following element is in
$\vec{F}_3\setminus\vec{F}_2$ and will be part of the generating set of $\vec{F}_3$.
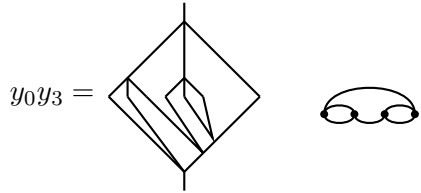
\[\begin{tikzpicture}[x=.5cm, y=.5cm,
    every edge/.style={
        draw,
      postaction={decorate,
                    decoration={markings}
                   }
        }
]

 \draw[thick] (1,0) -- (3,-2)--(5,0);
 \draw[thick] (1.5,0) -- (3,-2);
 \draw[thick] (2,0) -- (3.5,-1.5)--(2.5,0);
 \draw[thick] (3,0) -- (3.75,-1.1)--(3.5,0);
 
 \draw[thick] (1,0) -- (3,2)--(5,0);
\draw[thick] (1.5,.5) -- (1.5,0);
\draw[thick] (1.5,.5) -- (2,0);
\draw[thick] (3,0) -- (3,2);
\draw[thick] (2.5,0) -- (3,.5)--(3.5,0);
 
 \node at (-0.5,0) {$\scalebox{1}{$y_0y_3=$}$};

 \draw[thick] (3,2)--(3,2.5);
 \draw[thick] (3,-2)--(3,-2.5);

\end{tikzpicture}
\qquad 
\begin{tikzpicture}[x=.4cm, y=.4cm,
    every edge/.style={
        draw,
      postaction={decorate,
                    decoration={markings}
                   }
        }
]

\node at (2.5,-2.25) {$\scalebox{1}{}$};

\draw[thick] (1,0) to[out=90,in=90] (2,0);
\draw[thick] (1,0) to[out=90,in=90] (4,0); 
\draw[thick] (3,0) to[out=90,in=90] (4,0); 

\draw[thick] (1,0) to[out=-90,in=-90] (2,0); 
 \draw[thick] (2,0) to[out=-90,in=-90] (3,0);  
 \draw[thick] (3,0) to[out=-90,in=-90] (4,0);

\fill (1,0)  circle[radius=1.5pt];
\fill (2,0)  circle[radius=1.5pt];
\fill (3,0)  circle[radius=1.5pt];
\fill (4,0)  circle[radius=1.5pt];
  

\end{tikzpicture}
\]
\begin{align*}
y_0y_3(t)=
  \left\{ {\begin{array}{ll}
   .0\alpha & \textrm{ if $t=.00\alpha$} \\
   .1\alpha & \textrm{ if $t=.01\alpha$} \\
   .20\alpha & \textrm{ if $t=.02\alpha$} \\
   .21\alpha & \textrm{ if $t=.10\alpha$} \\
   .220\alpha & \textrm{ if $t=.11\alpha$} \\
   .221\alpha & \textrm{ if $t=.12\alpha$} \\
   .222\alpha & \textrm{ if $t=.2\alpha$} \\
  \end{array} } \right.
\end{align*}
This element does not belong to $\vec{F}_2$ and actually is not even in $F_2$. Indeed, by the very definition of $\iota$, 
an element of $\iota(F_2)$ (represented as a reduced tree diagram), cannot have a split attached below a middle edge.
\end{example}

Given a triadic rational expressed in ternary expansion $.a_1a_2\cdots a_n$, there is a corresponding path (starting from the root) 
in the tree of the standard triadic intervals. We denote by $c(.a_1a_2\cdots a_n)$ the weight of the end vertex  of this  path.
The next theorem is the main result of this section and its proof can be easily deduced from the preceding discussion. A different proof using maximality of $\vec{F}_3$ in a certain subgroup of $F_3$ will be given in Section 4 below.
\begin{theorem}\label{theo-stab}
The  ternary oriented Thompson group $\vec{F}_3$ is the stabilizer of the following subset of the triadic fractions
$$
Z:=\{ .a_1a_2\cdots a_n \; | \; \# \textrm{ of $1$'s is even}, c(.a_1a_2\cdots a_n) \textrm{ is even}\} \; .
$$
\end{theorem}

\section{A generating set for $\vec{F}_3$} 
The aim of this section is to exhibit a generating set for $\vec{F}_3$ and to show that it is finitely generated.
We recall that an element of $F_3$ is said to be positive if it can be expressed as the product of the generators $y_0, y_1, \ldots$ (but not of their inverses).
\begin{lemma}\label{pos-gen}
The subgroup $\vec{F}_3$ is generated by its positive elements. 
\end{lemma}
\begin{proof}
Let $(T_+,T_-)\in \vec{F}_3$. Possibly after adding pairs of opposing carets, we may suppose that the colouring of $\Gamma(T_+,T_-)$ is $+, -, + , -, \ldots$.
Now let $T$ be the tree only with merges attached to the right edges of other merges and with the same number of leaves as $T_\pm$. By construction, $\Gamma(T_+,T)$ and $\Gamma(T,T_-)$ are both $2$-colourable and thus $(T_+,T), (T,T_-)\in\vec{F}_3$. Clearly $(T_+,T)$ and $(T,T_-)^{-1}$ are positive. \\
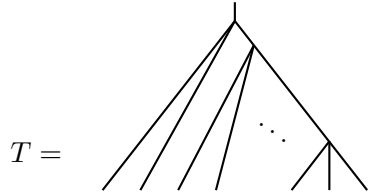
\[
\begin{tikzpicture}[x=.5cm, y=.5cm,
    every edge/.style={
        draw,
      postaction={decorate,
                    decoration={markings}
                   }
        }
]

\node at (-.75,1) {$\scalebox{1}{$T=$}$};

\draw[thick] (1,0) -- (4.5,4.5)--(8,0);
\draw[thick] (4.5,4.5)--(2,0);
\draw[thick] (5,3.85)--(3,0);
\draw[thick] (5,3.85)--(4,0);
\draw[thick] (7,1.3)--(6,0);
\draw[thick] (7,1.3)--(7,0);

\node at (5.5,1.75) {$\scalebox{1}{$\ddots$}$};

\draw[thick] (4.5,5)--(4.5,4.5);

\end{tikzpicture}
\]
\end{proof}
\begin{lemma}\label{lemmashift}
Consider the endomorphism  $\varphi_R : F_3\to  F_3$ (herein called the right-shift) defined graphically  as
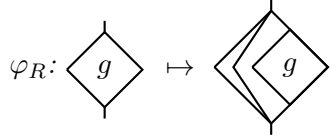
\[
\begin{tikzpicture}[x=1cm, y=1cm,
    every edge/.style={
        draw,
      postaction={decorate,
                    decoration={markings}
                   }
        }
]

\node at (-.45,0) {$\scalebox{1}{$\varphi_R$:}$};

\draw[thick] (0,0)--(.5,.5)--(1,0)--(.5,-.5)--(0,0);
\node at (1.5,0) {$\scalebox{1}{$\mapsto$}$};

\node at (.5,0) {$\scalebox{1}{$g$}$};
\node at (.5,.-.75) {$\scalebox{1}{}$};

 \draw[thick] (.5,.65)--(.5,.5);
 \draw[thick] (.5,-.65)--(.5,-.5);
 
\end{tikzpicture}
\begin{tikzpicture}[x=1cm, y=1cm,
    every edge/.style={
        draw,
      postaction={decorate,
                    decoration={markings}
                   }
        }
]
 
\node at (2.5,0) {$\scalebox{1}{$g$}$};

\draw[thick] (1.5,0)--(2.25,.75)--(3,0)--(2.25,-.75)--(1.5,0);
\draw[thick] (2,0)--(2.5,.5)--(3,0)--(2.5,-.5)--(2,0);
 
 \node at (1.5,.-.75) {$\scalebox{1}{}$};

\draw[thick] (2.25,.75)--(1.75,0)--(2.25,-.75);

 \draw[thick] (2.25,.75)--(2.25,.9);
 \draw[thick] (2.25,-.75)--(2.25,-.9);

\end{tikzpicture}
\]
Then $g\in\vec{F}_3$ if and only if $\varphi_R(g)\in\vec{F}_3$.
\end{lemma}
\begin{proof}
As   shown below, the right-shift simply adds a pair of parallel edges attached only to the left-most vertex of $\Gamma(g)$
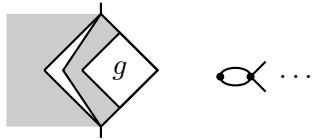
\[\begin{tikzpicture}[x=1cm, y=1cm,
    every edge/.style={
        draw,
      postaction={decorate,
                    decoration={markings}
                   }
        }
]

\node at (2.5,0) {$\scalebox{1}{$g$}$};
\fill [color=black,opacity=0.2] (1.5,0)--(2.25,.75)--(1,.75)--(1,-.75)--(2.25,-.75)--(1.5,0);
\fill [color=black,opacity=0.2] (2.25,.75)--(1.75,0)--(2.25,-.75)--(2.5,-.5)--(2,0)--(2.5,.5)--(2.25,.75);

\draw[thick] (1.5,0)--(2.25,.75)--(3,0)--(2.25,-.75)--(1.5,0);
\draw[thick] (2,0)--(2.5,.5)--(3,0)--(2.5,-.5)--(2,0);

\draw[thick] (2.25,.75)--(1.75,0)--(2.25,-.75);

 \draw[thick] (2.25,.75)--(2.25,.9);
 \draw[thick] (2.25,-.75)--(2.25,-.9);

\end{tikzpicture}
\qquad 
\begin{tikzpicture}[x=.4cm, y=.4cm,
    every edge/.style={
        draw,
      postaction={decorate,
                    decoration={markings}
                   }
        }
]

\node at (2.5,-1.75) {$\scalebox{1}{}$};

\node at (3.5,0) {$\scalebox{1}{$\ldots$}$};

\draw[thick] (2.5,-.5)--(2,0)--(2.5,.5);

\draw[thick] (1,0) to[out=90,in=90] (2,0);

\draw[thick] (1,0) to[out=-90,in=-90] (2,0);

\fill (1,0)  circle[radius=1.5pt];
\fill (2,0)  circle[radius=1.5pt]; 
\end{tikzpicture}
\]
\end{proof}
\begin{remark}
For all $i\geq 0$ we have that $\varphi_R(y_i)=y_{i+2}$.
\end{remark}
\begin{theorem}
The ternary oriented Thompson group $\vec{F}_3$ is generated by the elements 
\begin{align*}
& y_{2i+1}^2, y_{2i}y_{2i+2}, y_{2i}y_{2i+3}
\qquad i = 0, 1, 2.
\end{align*}
Moreover, the subgroups $\langle y_{2i+1}^2 \; | \; i =0,1,2 \rangle$, $\langle y_{2i}y_{2i+2} \; | \; i = 0,1,2 \rangle$, $\langle y_{2i}y_{2i+3}  \; | \; i = 0,1,2 \rangle$ are all isomorphic to $F_3$.
\end{theorem}
\begin{proof}
The proof of the main statement is divided into two parts. Firstly, we show that $\vec{F}_3$ is generated by the infinite family 
\begin{align}\label{generators}
& y_{2i+1}^2, y_{2i}y_{2i+2}, y_{2i}y_{2i+3}, 
\qquad i\geq 0.
\end{align}
Secondly, we    prove that   actually the first nine elements corresponding to $i=0,1,2$ are enough.

In the first place we want to show that all the elements in \eqref{generators} belong to $\vec{F}_3$. 
It was already observed in Remark \ref{rem1}
that the 
 elements $y_{2i}y_{2i+2}$ are in $\vec{F}_3$. 
 Thanks to Lemma \ref{lemmashift}, Example \ref{esey2} and Example \ref{ese3} we  know that $y_{2i+1}^2, y_{2i}y_{2i+3}\in\vec{F}_3$.

By Lemma \ref{pos-gen}   it is enough to show that   an arbitrary positive element $g$ in $\vec{F}_3$ can be expressed in terms of products of the above elements. 
We give a proof by induction on the number of   leaves of the tree diagram. The first case is when there are $3$ leaves. There is only one tree diagram with three leaves, namely the trivial element, so we may suppose that the claim is true for a tree diagram with $3+2n$ leaves ($n\geq 0$) and prove the claim for $3+2(n+1)$.

In the top tree, one of the following 14 sub-trees must occur.  Actually some of them (namely the cases 2, 3, 4, 5, 8, 9, 10, 11) may be ruled out because the corresponding subgraph of $\Gamma$ does not admit a colouring $+, -, +, -, \ldots$. 
 In the other cases, it is possible to multiply $g$ by the inverse of a suitable element in the aforementioned family  
 which we denote by $h$ and reduce the number of leaves of the tree diagram. Below are shown the elements $h$ to consider and the subgraph in red is the part that will vanish in $gh^{-1}$.
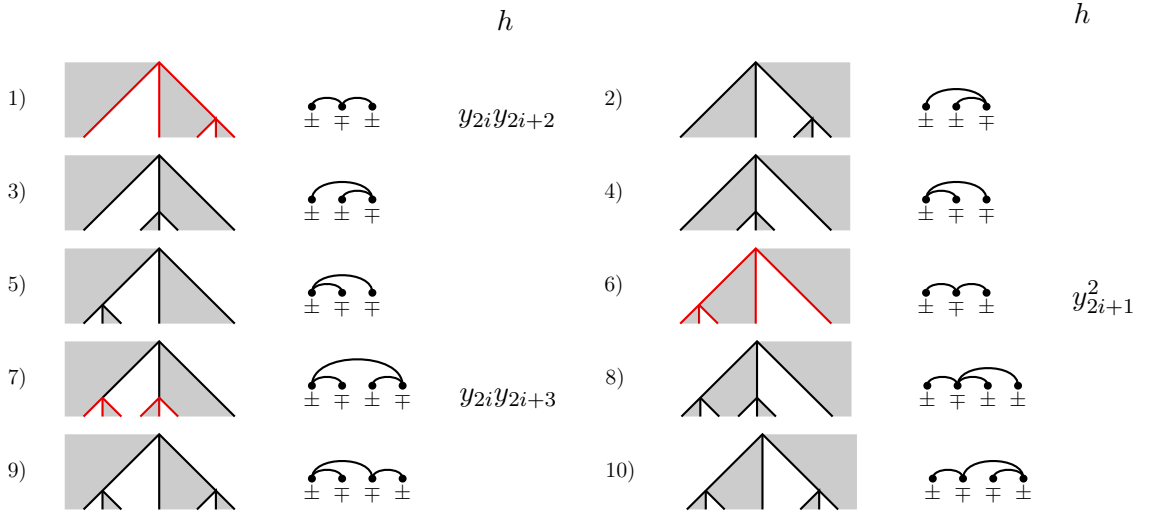
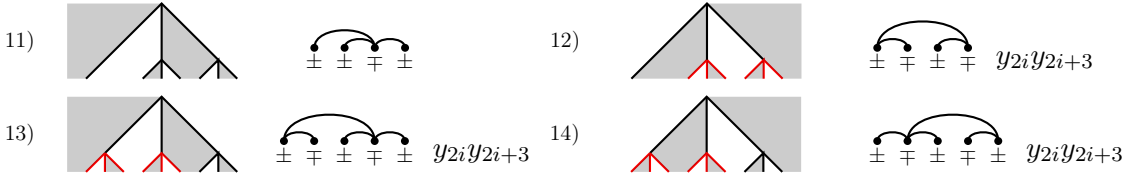
\[
\begin{tikzpicture}[x=.5cm, y=.5cm,
    every edge/.style={
        draw,
      postaction={decorate,
                    decoration={markings}
                   }
        }
]

 \node at (-.75,0) {$\scalebox{1}{\phantom{$\;$}}$};
 \node at (5,0) {$\scalebox{1}{\phantom{$\;$}}$};

\end{tikzpicture}
\qquad
\begin{tikzpicture}[x=.4cm, y=.4cm,
    every edge/.style={
        draw,
      postaction={decorate,
                    decoration={markings}
                   }
        }
]

 \node at (1,0) {$\scalebox{1}{\phantom{$\;$}}$};
 \node at (3,0) {$\scalebox{1}{\phantom{$\;$}}$};

\end{tikzpicture}
\qquad
\begin{tikzpicture}[x=.4cm, y=.4cm,
    every edge/.style={
        draw,
      postaction={decorate,
                    decoration={markings}
                   }
        }
]

 \node at (1,0) {$\scalebox{1}{$\phantom{y_{2i}y_{2i+2}}$}$};
 \node at (1,0) {$\scalebox{1}{$h$}$};

\end{tikzpicture}
\quad
\begin{tikzpicture}[x=.5cm, y=.5cm,
    every edge/.style={
        draw,
      postaction={decorate,
                    decoration={markings}
                   }
        }
]

 \node at (-.75,0) {$\scalebox{1}{\phantom{$\;$}}$};
 \node at (5,0) {$\scalebox{1}{\phantom{$\;$}}$};

\end{tikzpicture}
\qquad
\begin{tikzpicture}[x=.4cm, y=.4cm,
    every edge/.style={
        draw,
      postaction={decorate,
                    decoration={markings}
                   }
        }
]

 \node at (1,0) {$\scalebox{1}{\phantom{$\;$}}$};
 \node at (3,0) {$\scalebox{1}{\phantom{$\;$}}$};
 
\end{tikzpicture}
\qquad
\begin{tikzpicture}[x=.4cm, y=.4cm,
    every edge/.style={
        draw,
      postaction={decorate,
                    decoration={markings}
                   }
        }
]
 
 \node at (1,0) {$\scalebox{1}{$h$}$};

 \node at (1,0) {$\scalebox{1}{\phantom{$y_{2i+1}^2$}}$};

\end{tikzpicture}
\]
\[
\begin{tikzpicture}[x=.5cm, y=.5cm,
    every edge/.style={
        draw,
      postaction={decorate,
                    decoration={markings}
                   }
        }
]

\node at (-.75,1) {$\scalebox{.75}{1)}$};

\draw[thick, red] (1,0) -- (3,2)--(5,0);
 \draw[thick, red] (3,0) -- (3,2);
\draw[thick, red] (4,0) -- (4.5,.5)--(4.5,0);

\fill [color=black,opacity=0.2]   (.5,0)--(.5,2)--(3,2)--(1,0); 
 \fill [color=black,opacity=0.2]   (4.5,0)--(4.5,.5)--(5,0); 
\fill [color=black,opacity=0.2]   (3,0)--(3,2)--(4.5,0.5)--(4,0);

\end{tikzpicture}
\qquad
\begin{tikzpicture}[x=.4cm, y=.4cm,
    every edge/.style={
        draw,
      postaction={decorate,
                    decoration={markings}
                   }
        }
]

\draw[thick] (1,0) to[out=90,in=90] (2,0);
\draw[thick] (2,0) to[out=90,in=90] (3,0);

\fill (1,0)  circle[radius=1.5pt];
\fill (2,0)  circle[radius=1.5pt];
\fill (3,0)  circle[radius=1.5pt];
 
 \node at (1,-.5) {$\scalebox{.75}{$\pm$}$};
\node at (2,-.5) {$\scalebox{.75}{$\mp$}$};
\node at (3,-.5) {$\scalebox{.75}{$\pm$}$};

\end{tikzpicture}
\qquad
\begin{tikzpicture}[x=.4cm, y=.4cm,
    every edge/.style={
        draw,
      postaction={decorate,
                    decoration={markings}
                   }
        }
]

 \node at (1,0) {$\scalebox{1}{$y_{2i}y_{2i+2}$}$}; 

\end{tikzpicture}
\quad
\begin{tikzpicture}[x=.5cm, y=.5cm,
    every edge/.style={
        draw,
      postaction={decorate,
                    decoration={markings}
                   }
        }
]

\node at (-.75,1) {$\scalebox{.75}{2)}$};

\draw[thick] (1,0) -- (3,2)--(5,0);
 \draw[thick] (3,0) -- (3,2);
\draw[thick] (4,0) -- (4.5,.5)--(4.5,0);

 \fill [color=black,opacity=0.2]   (4,0)--(4.5,.5)--(4.5,0)--(4,0); 
\fill [color=black,opacity=0.2]   (1,0)--(3,2)--(3,0)--(2.5,0)--(2,0); 
\fill [color=black,opacity=0.2]   (3,2)--(5.5,2)--(5.5,0)--(5,0);

\end{tikzpicture}
\qquad
\begin{tikzpicture}[x=.4cm, y=.4cm,
    every edge/.style={
        draw,
      postaction={decorate,
                    decoration={markings}
                   }
        }
]

\draw[thick] (1,0) to[out=90,in=90] (3,0);
\draw[thick] (2,0) to[out=90,in=90] (3,0);

\fill (1,0)  circle[radius=1.5pt];
\fill (2,0)  circle[radius=1.5pt];
\fill (3,0)  circle[radius=1.5pt];
 
 \node at (1,-.5) {$\scalebox{.75}{$\pm$}$};
\node at (2,-.5) {$\scalebox{.75}{$\pm$}$};
\node at (3,-.5) {$\scalebox{.75}{$\mp$}$};

\end{tikzpicture}
\qquad
\begin{tikzpicture}[x=.4cm, y=.4cm,
    every edge/.style={
        draw,
      postaction={decorate,
                    decoration={markings}
                   }
        }
]

 \node at (1,0) {$\scalebox{1}{\phantom{$y_{2i+1}^2$}}$};

\end{tikzpicture}
\]
\[
\begin{tikzpicture}[x=.5cm, y=.5cm,
    every edge/.style={
        draw,
      postaction={decorate,
                    decoration={markings}
                   }
        }
]

\node at (-.75,1) {$\scalebox{.75}{3)}$};

\draw[thick] (1,0) -- (3,2)--(5,0);
 \draw[thick] (3,0) -- (3,2);
\draw[thick] (2.5,0) -- (3,.5)--(3.5,0);

\fill [color=black,opacity=0.2]   (.5,0)--(.5,2)--(3,2)--(1,0); 
\fill [color=black,opacity=0.2]   (2.5,0)--(3,.5)--(3,0); 
\fill [color=black,opacity=0.2]   (3,2)--(5,0)--(3.5,0)--(3,0.5)--(3,2);

\end{tikzpicture}
\qquad
\begin{tikzpicture}[x=.4cm, y=.4cm,
    every edge/.style={
        draw,
      postaction={decorate,
                    decoration={markings}
                   }
        }
]

\draw[thick] (1,0) to[out=90,in=90] (3,0);
\draw[thick] (2,0) to[out=90,in=90] (3,0);

\fill (1,0)  circle[radius=1.5pt];
\fill (2,0)  circle[radius=1.5pt];
\fill (3,0)  circle[radius=1.5pt];
 
 \node at (1,-.5) {$\scalebox{.75}{$\pm$}$};
\node at (2,-.5) {$\scalebox{.75}{$\pm$}$};
\node at (3,-.5) {$\scalebox{.75}{$\mp$}$};

\end{tikzpicture}
\qquad
\begin{tikzpicture}[x=.4cm, y=.4cm,
    every edge/.style={
        draw,
      postaction={decorate,
                    decoration={markings}
                   }
        }
]
 
 \node at (1,0) {$\scalebox{1}{\phantom{$y_{2i}y_{2i+2}$}}$};

\end{tikzpicture}
\quad
\begin{tikzpicture}[x=.5cm, y=.5cm,
    every edge/.style={
        draw,
      postaction={decorate,
                    decoration={markings}
                   }
        }
]

\node at (-.75,1) {$\scalebox{.75}{4)}$};

\draw[thick] (1,0) -- (3,2)--(5,0); 
\draw[thick] (3,0) -- (3,2);
\draw[thick] (2.5,0) -- (3,.5)--(3.5,0);

 \fill [color=black,opacity=0.2]   (3.5,0)--(3,.5)--(3,0)--(3.5,0); 
\fill [color=black,opacity=0.2]   (1,0)--(1.5,.5)--(3,2)--(3,.5)--(2.5,0)--(2,0); 
\fill [color=black,opacity=0.2]   (3,2)--(5.5,2)--(5.5,0)--(5,0);

\end{tikzpicture}
\qquad
\begin{tikzpicture}[x=.4cm, y=.4cm,
    every edge/.style={
        draw,
      postaction={decorate,
                    decoration={markings}
                   }
        }
]

\draw[thick] (1,0) to[out=90,in=90] (3,0);
\draw[thick] (1,0) to[out=90,in=90] (2,0);

\fill (1,0)  circle[radius=1.5pt];
\fill (2,0)  circle[radius=1.5pt];
\fill (3,0)  circle[radius=1.5pt];
 
 \node at (1,-.5) {$\scalebox{.75}{$\pm$}$};
\node at (2,-.5) {$\scalebox{.75}{$\mp$}$};
\node at (3,-.5) {$\scalebox{.75}{$\mp$}$};

\end{tikzpicture}
\qquad
\begin{tikzpicture}[x=.4cm, y=.4cm,
    every edge/.style={
        draw,
      postaction={decorate,
                    decoration={markings}
                   }
        }
]

 \node at (1,0) {$\scalebox{1}{\phantom{$y_{2i+1}^2$}}$};

\end{tikzpicture}
\]
\[
\begin{tikzpicture}[x=.5cm, y=.5cm,
    every edge/.style={
        draw,
      postaction={decorate,
                    decoration={markings}
                   }
        }
]

\node at (-.75,1) {$\scalebox{.75}{5)}$};

\draw[thick] (1,0) -- (3,2)--(5,0);
\draw[thick] (1.5,.5) -- (1.5,0);
\draw[thick] (1.5,.5) -- (2,0);
\draw[thick] (3,0) -- (3,2);
 
\fill [color=black,opacity=0.2]   (.5,0)--(.5,2)--(3,2)--(1,0); 
\fill [color=black,opacity=0.2]   (1.5,0)--(1.5,.5)--(2,0); 
\fill [color=black,opacity=0.2]   (3,2)--(5,0)--(3,0)--(3,0.5)--(3,2);

\end{tikzpicture}
\qquad
\begin{tikzpicture}[x=.4cm, y=.4cm,
    every edge/.style={
        draw,
      postaction={decorate,
                    decoration={markings}
                   }
        }
]

\draw[thick] (1,0) to[out=90,in=90] (3,0);
\draw[thick] (1,0) to[out=90,in=90] (2,0);

\fill (1,0)  circle[radius=1.5pt];
\fill (2,0)  circle[radius=1.5pt];
\fill (3,0)  circle[radius=1.5pt];
 
 \node at (1,-.5) {$\scalebox{.75}{$\pm$}$};
\node at (2,-.5) {$\scalebox{.75}{$\mp$}$};
\node at (3,-.5) {$\scalebox{.75}{$\mp$}$};

\end{tikzpicture}
\qquad
\begin{tikzpicture}[x=.4cm, y=.4cm,
    every edge/.style={
        draw,
      postaction={decorate,
                    decoration={markings}
                   }
        }
]

 \node at (1,0) {$\scalebox{1}{\phantom{$y_{2i}y_{2i+2}$}}$};

\end{tikzpicture}
\quad
\begin{tikzpicture}[x=.5cm, y=.5cm,
    every edge/.style={
        draw,
      postaction={decorate,
                    decoration={markings}
                   }
        }
]

\node at (-.75,1) {$\scalebox{.75}{6)}$};

\draw[thick, red] (1,0) -- (3,2)--(5,0);
\draw[thick, red] (1.5,.5) -- (1.5,0);
\draw[thick, red] (1.5,.5) -- (2,0);
\draw[thick, red] (3,0) -- (3,2);
 
\fill [color=black,opacity=0.2]   (1,0)--(1.5,.5)--(1.5,0)--(1,0); 
 \fill [color=black,opacity=0.2]   (2,0)--(1.5,.5)--(3,2)--(3,0)--(2.5,0)--(2,0); 
\fill [color=black,opacity=0.2]   (3,2)--(5.5,2)--(5.5,0)--(5,0);

\end{tikzpicture}
\qquad
\begin{tikzpicture}[x=.4cm, y=.4cm,
    every edge/.style={
        draw,
      postaction={decorate,
                    decoration={markings}
                   }
        }
]

\draw[thick] (1,0) to[out=90,in=90] (2,0);
\draw[thick] (2,0) to[out=90,in=90] (3,0);

\fill (1,0)  circle[radius=1.5pt];
\fill (2,0)  circle[radius=1.5pt];
\fill (3,0)  circle[radius=1.5pt];
 
 \node at (1,-.5) {$\scalebox{.75}{$\pm$}$};
\node at (2,-.5) {$\scalebox{.75}{$\mp$}$};
\node at (3,-.5) {$\scalebox{.75}{$\pm$}$};

\end{tikzpicture}
\qquad
\begin{tikzpicture}[x=.4cm, y=.4cm,
    every edge/.style={
        draw,
      postaction={decorate,
                    decoration={markings}
                   }
        }
]

 \node at (1,0) {$\scalebox{1}{$y_{2i+1}^2$}$};

\end{tikzpicture}
\]
\[
\begin{tikzpicture}[x=.5cm, y=.5cm,
    every edge/.style={
        draw,
      postaction={decorate,
                    decoration={markings}
                   }
        }
]

\node at (-.75,1) {$\scalebox{.75}{7)}$};

\draw[thick] (1.5,0.5) -- (3,2)--(5,0);
\draw[thick, red] (1.5,.5) -- (1,0);
\draw[thick, red] (1.5,.5) -- (1.5,0);
\draw[thick, red] (1.5,.5) -- (2,0);
\draw[thick] (3,0.5) -- (3,2);
\draw[thick, red] (2.5,0) -- (3,.5)--(3.5,0);
\draw[thick, red] (3,0) -- (3,.5);

\fill [color=black,opacity=0.2]   (.5,0)--(.5,2)--(3,2)--(1,0); 
\fill [color=black,opacity=0.2]   (1.5,0)--(1.5,.5)--(2,0); 
\fill [color=black,opacity=0.2]   (3,2)--(5,0)--(3.5,0)--(3,0.5)--(3,2); 
\fill [color=black,opacity=0.2]   (2.5,0)--(3,.5)--(3,0);

\end{tikzpicture}
\qquad
\begin{tikzpicture}[x=.4cm, y=.4cm,
    every edge/.style={
        draw,
      postaction={decorate,
                    decoration={markings}
                   }
        }
]

\draw[thick] (1,0) to[out=90,in=90] (4,0);
\draw[thick] (1,0) to[out=90,in=90] (2,0); 
\draw[thick] (3,0) to[out=90,in=90] (4,0);

\fill (1,0)  circle[radius=1.5pt];
\fill (2,0)  circle[radius=1.5pt];
\fill (3,0)  circle[radius=1.5pt];
\fill (4,0)  circle[radius=1.5pt];
 
 \node at (1,-.5) {$\scalebox{.75}{$\pm$}$};
\node at (2,-.5) {$\scalebox{.75}{$\mp$}$};
\node at (3,-.5) {$\scalebox{.75}{$\pm$}$};
\node at (4,-.5) {$\scalebox{.75}{$\mp$}$};

\end{tikzpicture}
\quad
\begin{tikzpicture}[x=.4cm, y=.4cm,
    every edge/.style={
        draw,
      postaction={decorate,
                    decoration={markings}
                   }
        }
]

 \node at (1,0) {$\scalebox{1}{$y_{2i}y_{2i+3}$}$};

\end{tikzpicture} 
\quad
\begin{tikzpicture}[x=.5cm, y=.5cm,
    every edge/.style={
        draw,
      postaction={decorate,
                    decoration={markings}
                   }
        }
]

\node at (-.75,1) {$\scalebox{.75}{8)}$};

\draw[thick] (1,0) -- (3,2)--(5,0);
\draw[thick] (1.5,.5) -- (1.5,0);
\draw[thick] (1.5,.5) -- (2,0);
\draw[thick] (3,0) -- (3,2);
\draw[thick] (2.5,0) -- (3,.5)--(3.5,0);

\fill [color=black,opacity=0.2]   (1,0)--(1.5,.5)--(1.5,0)--(1,0); 
\fill [color=black,opacity=0.2]   (3.5,0)--(3,.5)--(3,0)--(3.5,0); 
\fill [color=black,opacity=0.2]   (2,0)--(1.5,.5)--(3,2)--(3,.5)--(2.5,0)--(2,0); 
\fill [color=black,opacity=0.2]   (3,2)--(5.5,2)--(5.5,0)--(5,0);

\end{tikzpicture}
\qquad
\begin{tikzpicture}[x=.4cm, y=.4cm,
    every edge/.style={
        draw,
      postaction={decorate,
                    decoration={markings}
                   }
        }
]

\draw[thick] (2,0) to[out=90,in=90] (4,0);
\draw[thick] (1,0) to[out=90,in=90] (2,0); 
\draw[thick] (2,0) to[out=90,in=90] (3,0);

\fill (1,0)  circle[radius=1.5pt];
\fill (2,0)  circle[radius=1.5pt];
\fill (3,0)  circle[radius=1.5pt];
\fill (4,0)  circle[radius=1.5pt];
 
 \node at (1,-.5) {$\scalebox{.75}{$\pm$}$};
\node at (2,-.5) {$\scalebox{.75}{$\mp$}$};
\node at (3,-.5) {$\scalebox{.75}{$\pm$}$};
\node at (4,-.5) {$\scalebox{.75}{$\pm$}$};

\end{tikzpicture}
\qquad
\begin{tikzpicture}[x=.4cm, y=.4cm,
    every edge/.style={
        draw,
      postaction={decorate,
                    decoration={markings}
                   }
        }
]

 \node at (1,0) {$\scalebox{1}{\phantom{$y_{2i+1}^2$}}$};

\end{tikzpicture}
\]
\[
\begin{tikzpicture}[x=.5cm, y=.5cm,
    every edge/.style={
        draw,
      postaction={decorate,
                    decoration={markings}
                   }
        }
]

\node at (-.75,1) {$\scalebox{.75}{9)}$};

\draw[thick] (1,0) -- (3,2)--(5,0);
\draw[thick] (1.5,.5) -- (1.5,0);
\draw[thick] (1.5,.5) -- (2,0);
\draw[thick] (3,0) -- (3,2);
\draw[thick] (4,0) -- (4.5,.5)--(4.5,0);

\fill [color=black,opacity=0.2]   (.5,0)--(.5,2)--(3,2)--(1,0); 
\fill [color=black,opacity=0.2]   (1.5,0)--(1.5,.5)--(2,0); 
\fill [color=black,opacity=0.2]   (4.5,0)--(4.5,.5)--(5,0); 
\fill [color=black,opacity=0.2]   (3,2)--(4.5,.5)--(4,0)--(3,0)--(3,2);

\end{tikzpicture}
\qquad
\begin{tikzpicture}[x=.4cm, y=.4cm,
    every edge/.style={
        draw,
      postaction={decorate,
                    decoration={markings}
                   }
        }
]

\draw[thick] (1,0) to[out=90,in=90] (3,0);
\draw[thick] (1,0) to[out=90,in=90] (2,0); 
\draw[thick] (3,0) to[out=90,in=90] (4,0);

\fill (1,0)  circle[radius=1.5pt];
\fill (2,0)  circle[radius=1.5pt];
\fill (3,0)  circle[radius=1.5pt];
\fill (4,0)  circle[radius=1.5pt];
 
 \node at (1,-.5) {$\scalebox{.75}{$\pm$}$};
\node at (2,-.5) {$\scalebox{.75}{$\mp$}$};
\node at (3,-.5) {$\scalebox{.75}{$\mp$}$};
\node at (4,-.5) {$\scalebox{.75}{$\pm$}$};

\end{tikzpicture}
\quad
\begin{tikzpicture}[x=.4cm, y=.4cm,
    every edge/.style={
        draw,
      postaction={decorate,
                    decoration={markings}
                   }
        }
] 

 \node at (1,0) {$\scalebox{1}{\phantom{$y_{2i}y_{2i+2}$}}$};

\end{tikzpicture}
\quad
\begin{tikzpicture}[x=.5cm, y=.5cm,
    every edge/.style={
        draw,
      postaction={decorate,
                    decoration={markings}
                   }
        }
]

\node at (-.75,1) {$\scalebox{.75}{10)}$};

\draw[thick] (1,0) -- (3,2)--(5,0);
\draw[thick] (1.5,.5) -- (1.5,0);
\draw[thick] (1.5,.5) -- (2,0);
\draw[thick] (3,0) -- (3,2);
\draw[thick] (4,0) -- (4.5,.5)--(4.5,0);

\fill [color=black,opacity=0.2]   (1,0)--(1.5,.5)--(1.5,0)--(1,0); 
\fill [color=black,opacity=0.2]   (4,0)--(4.5,.5)--(4.5,0)--(4,0); 
\fill [color=black,opacity=0.2]   (2,0)--(1.5,.5)--(3,2)--(3,0)--(2.5,0)--(2,0); 
\fill [color=black,opacity=0.2]   (3,2)--(5.5,2)--(5.5,0)--(5,0);

\end{tikzpicture}
\qquad
\begin{tikzpicture}[x=.4cm, y=.4cm,
    every edge/.style={
        draw,
      postaction={decorate,
                    decoration={markings}
                   }
        }
]

\draw[thick] (1,0) to[out=90,in=90] (2,0);
\draw[thick] (2,0) to[out=90,in=90] (4,0); 
\draw[thick] (3,0) to[out=90,in=90] (4,0);

\fill (1,0)  circle[radius=1.5pt];
\fill (2,0)  circle[radius=1.5pt];
\fill (3,0)  circle[radius=1.5pt];
\fill (4,0)  circle[radius=1.5pt];
 
 \node at (1,-.5) {$\scalebox{.75}{$\pm$}$};
\node at (2,-.5) {$\scalebox{.75}{$\mp$}$};
\node at (3,-.5) {$\scalebox{.75}{$\mp$}$};
\node at (4,-.5) {$\scalebox{.75}{$\pm$}$};

\end{tikzpicture}
\qquad
\begin{tikzpicture}[x=.4cm, y=.4cm,
    every edge/.style={
        draw,
      postaction={decorate,
                    decoration={markings}
                   }
        }
]

 \node at (1,0) {$\scalebox{1}{\phantom{$y_{2i+1}^2$}}$};

\end{tikzpicture}
\]
\[
\begin{tikzpicture}[x=.5cm, y=.5cm,
    every edge/.style={
        draw,
      postaction={decorate,
                    decoration={markings}
                   }
        }
]

\node at (-.75,1) {$\scalebox{.75}{11)}$};

\draw[thick] (1,0) -- (3,2)--(5,0);
\draw[thick] (2.5,0) -- (3,.5)--(3.5,0);
\draw[thick] (3,0) -- (3,2);
\draw[thick] (4,0) -- (4.5,.5)--(4.5,0);

\fill [color=black,opacity=0.2]   (.5,0)--(.5,2)--(3,2)--(1,0); 
\fill [color=black,opacity=0.2]   (2.5,0)--(3,.5)--(3,0); 
\fill [color=black,opacity=0.2]   (4.5,0)--(4.5,.5)--(5,0); 
\fill [color=black,opacity=0.2]   (3.5,0)--(3,.5)--(3,2)--(4.5,0.5)--(4,0);

\end{tikzpicture}
\qquad
\begin{tikzpicture}[x=.4cm, y=.4cm,
    every edge/.style={
        draw,
      postaction={decorate,
                    decoration={markings}
                   }
        }
]

\draw[thick] (1,0) to[out=90,in=90] (3,0);
\draw[thick] (2,0) to[out=90,in=90] (3,0); 
\draw[thick] (3,0) to[out=90,in=90] (4,0);

\fill (1,0)  circle[radius=1.5pt];
\fill (2,0)  circle[radius=1.5pt];
\fill (3,0)  circle[radius=1.5pt];
\fill (4,0)  circle[radius=1.5pt];
 
 \node at (1,-.5) {$\scalebox{.75}{$\pm$}$};
\node at (2,-.5) {$\scalebox{.75}{$\pm$}$};
\node at (3,-.5) {$\scalebox{.75}{$\mp$}$};
\node at (4,-.5) {$\scalebox{.75}{$\pm$}$};

\end{tikzpicture}
\begin{tikzpicture}[x=.4cm, y=.4cm,
    every edge/.style={
        draw,
      postaction={decorate,
                    decoration={markings}
                   }
        }
]

 \node at (1,0) {$\scalebox{1}{\phantom{$y_{2i}y_{2i+2}$}}$};

\end{tikzpicture}
\begin{tikzpicture}[x=.5cm, y=.5cm,
    every edge/.style={
        draw,
      postaction={decorate,
                    decoration={markings}
                   }
        }
]

\node at (-.75,1) {$\scalebox{.75}{12)}$};

\draw[thick] (1,0) -- (3,2)--(4.5,0.5);
\draw[thick, red] (2.5,0) -- (3,.5)--(3.5,0);
\draw[thick] (3,0.5) -- (3,2);
\draw[thick, red] (3,0) -- (3,.5);
\draw[thick, red] (4,0) -- (4.5,.5)--(4.5,0);
\draw[thick, red] (4.5,.5)--(5,0);

\fill [color=black,opacity=0.2]   (3.5,0)--(3,.5)--(3,0); 
\fill [color=black,opacity=0.2]   (4,0)--(4.5,.5)--(4.5,0)--(4,0); 
\fill [color=black,opacity=0.2]   (1,0)--(3,2)--(3,.5)--(2.5,0)--(1,0); 
\fill [color=black,opacity=0.2]   (3,2)--(5.5,2)--(5.5,0)--(5,0);

\end{tikzpicture}
\qquad
\begin{tikzpicture}[x=.4cm, y=.4cm,
    every edge/.style={
        draw,
      postaction={decorate,
                    decoration={markings}
                   }
        }
]

\draw[thick] (1,0) to[out=90,in=90] (4,0);
\draw[thick] (1,0) to[out=90,in=90] (2,0); 
\draw[thick] (3,0) to[out=90,in=90] (4,0);

\fill (1,0)  circle[radius=1.5pt];
\fill (2,0)  circle[radius=1.5pt];
\fill (3,0)  circle[radius=1.5pt];
\fill (4,0)  circle[radius=1.5pt];
 
 \node at (1,-.5) {$\scalebox{.75}{$\pm$}$};
\node at (2,-.5) {$\scalebox{.75}{$\mp$}$};
\node at (3,-.5) {$\scalebox{.75}{$\pm$}$};
\node at (4,-.5) {$\scalebox{.75}{$\mp$}$};

\end{tikzpicture}
\begin{tikzpicture}[x=.4cm, y=.4cm,
    every edge/.style={
        draw,
      postaction={decorate,
                    decoration={markings}
                   }
        }
]

 \node at (1,0) {$\scalebox{1}{$y_{2i}y_{2i+3}$}$};

\end{tikzpicture}
\]
\[
\begin{tikzpicture}[x=.5cm, y=.5cm,
    every edge/.style={
        draw,
      postaction={decorate,
                    decoration={markings}
                   }
        }
]

\node at (-.75,1) {$\scalebox{.75}{13)}$};

\draw[thick] (1.5,.5) -- (3,2)--(5,0);
 \draw[thick] (3,0.5) -- (3,2);
\draw[thick] (4,0) -- (4.5,.5)--(4.5,0);
 
\draw[thick, red] (2.5,0) -- (3,0.5)--(3.5,0);
\draw[thick, red] (3,0)--(3,.5);

\draw[thick, red] (1.5,.5) -- (1,0);
\draw[thick, red] (1.5,.5) -- (1.5,0);
\draw[thick, red] (1.5,.5) -- (2,0);

\fill [color=black,opacity=0.2]   (.5,0)--(.5,2)--(3,2)--(1,0); 
\fill [color=black,opacity=0.2]   (2.5,0)--(3,.5)--(3,0); 
\fill [color=black,opacity=0.2]   (1.5,0)--(1.5,.5)--(2,0); 
\fill [color=black,opacity=0.2]   (4.5,0)--(4.5,.5)--(5,0); 
\fill [color=black,opacity=0.2]   (3.5,0)--(3,.5)--(3,2)--(4.5,0.5)--(4,0); 

\node at (5.5,1) {$\scalebox{.75}{}$};

\end{tikzpicture}
\begin{tikzpicture}[x=.4cm, y=.4cm,
    every edge/.style={
        draw,
      postaction={decorate,
                    decoration={markings}
                   }
        }
]

\draw[thick] (1,0) to[out=90,in=90] (2,0);
\draw[thick] (1,0) to[out=90,in=90] (4,0);
\draw[thick] (3,0) to[out=90,in=90] (4,0); 
\draw[thick] (4,0) to[out=90,in=90] (5,0);

\fill (1,0)  circle[radius=1.5pt];
\fill (2,0)  circle[radius=1.5pt];
\fill (3,0)  circle[radius=1.5pt];
\fill (4,0)  circle[radius=1.5pt];
\fill (5,0)  circle[radius=1.5pt];
 
 \node at (1,-.5) {$\scalebox{.75}{$\pm$}$};
\node at (2,-.5) {$\scalebox{.75}{$\mp$}$};
\node at (3,-.5) {$\scalebox{.75}{$\pm$}$};
\node at (4,-.5) {$\scalebox{.75}{$\mp$}$};
\node at (5,-.5) {$\scalebox{.75}{$\pm$}$};

\end{tikzpicture}
\begin{tikzpicture}[x=.4cm, y=.4cm,
    every edge/.style={
        draw,
      postaction={decorate,
                    decoration={markings}
                   }
        }
]

 \node at (1,0) {$\scalebox{1}{$y_{2i}y_{2i+3}$}$};

\end{tikzpicture}
\begin{tikzpicture}[x=.5cm, y=.5cm,
    every edge/.style={
        draw,
      postaction={decorate,
                    decoration={markings}
                   }
        }
]

\node at (-.75,1) {$\scalebox{.75}{14)}$};

\draw[thick] (1.5,.5) -- (3,2)--(5,0);
 \draw[thick] (3,0.5) -- (3,2);
\draw[thick] (4,0) -- (4.5,.5)--(4.5,0);
 
\draw[thick, red] (2.5,0) -- (3,0.5)--(3.5,0);
\draw[thick, red] (3,0)--(3,.5);

\draw[thick, red] (1.5,.5) -- (1,0);
\draw[thick, red] (1.5,.5) -- (1.5,0);
\draw[thick, red] (1.5,.5) -- (2,0);

\fill [color=black,opacity=0.2]   (1,0)--(1.5,.5)--(1.5,0); 
\fill [color=black,opacity=0.2]   (3.5,0)--(3,.5)--(3,0); 
\fill [color=black,opacity=0.2]   (4,0)--(4.5,.5)--(4.5,0)--(4,0); 
\fill [color=black,opacity=0.2]   (2,0)--(1.5,0.5)--(3,2)--(3,.5)--(2.5,0)--(1,0); 
\fill [color=black,opacity=0.2]   (3,2)--(5.5,2)--(5.5,0)--(5,0);

\end{tikzpicture}
\qquad
\begin{tikzpicture}[x=.4cm, y=.4cm,
    every edge/.style={
        draw,
      postaction={decorate,
                    decoration={markings}
                   }
        }
]

\draw[thick] (1,0) to[out=90,in=90] (2,0);
\draw[thick] (2,0) to[out=90,in=90] (5,0);
\draw[thick] (2,0) to[out=90,in=90] (3,0);
\draw[thick] (4,0) to[out=90,in=90] (5,0);

\fill (1,0)  circle[radius=1.5pt];
\fill (2,0)  circle[radius=1.5pt];
\fill (3,0)  circle[radius=1.5pt];
\fill (4,0)  circle[radius=1.5pt]; 
\fill (5,0)  circle[radius=1.5pt];
 
 \node at (1,-.5) {$\scalebox{.75}{$\pm$}$};
\node at (2,-.5) {$\scalebox{.75}{$\mp$}$};
\node at (3,-.5) {$\scalebox{.75}{$\pm$}$};
\node at (4,-.5) {$\scalebox{.75}{$\mp$}$};
\node at (5,-.5) {$\scalebox{.75}{$\pm$}$};

\end{tikzpicture}
\begin{tikzpicture}[x=.4cm, y=.4cm,
    every edge/.style={
        draw,
      postaction={decorate,
                    decoration={markings}
                   }
        }
]

 \node at (1,0) {$\scalebox{1}{$y_{2i}y_{2i+3}$}$};

\end{tikzpicture}
\] 

We now prove that $\vec{F}_3$ is finitely generated thanks to the defining relations of $F_3$.  
Indeed, for $i\geq 1$ it holds 
$(y_{0}y_{2})^{-1}(y_{2i}y_{2i+2})(y_{0}y_{2})=y_{2(i+2)}y_{2(i+2)+2}$,
$y_{1}^{-2}y_{2i+1}^2y_{1}^2=y_{2(i+2)+1}^2$, $(y_{0}y_{3})^{-1}(y_{2i}y_{2i+3})(y_{0}y_{3})=y_{2(i+2)}y_{2(i+2)+3}$ and thus we see that the elements in \eqref{generators} corresponding to $i=0, 1, 2$ generate the whole subgroup.

We now take care of the last claim of this theorem. Since every proper homomorphic image of $F_3$ is abelian \cite[Theorem 4.13]{Brown},
it is enough to check that the elements $u_i:=y_{2i+1}^2$, $v_i:=y_{2i}y_{2i+2}$, $w_i:=y_{2i}y_{2i+3}$ satisfy the generating relations of $F_3$.
This can be done  easily, but 
we omit tedious 
computations. 
\end{proof}

\section{The oriented subgroup $\vec{F}_3$ in $F_3$} 
The aim of this section is to show that the oriented subgroup $\vec {F}_3$ gives rise to a maximal subgroup of infinite index in $F_3$ isomorphic to $\vec{F}_3$ that does not fix any point in $(0,1)$. The situation is therefore similar to that of $\vec F \leq \vec{F}_2 \leq \ F_2 = F$ studied by Golan and Sapir in \cite{GS2}.
 More precisely, we show that $\vec{F}_3$ is maximal in a subgroup of index $2$ in $F_3$ that is isomorphic to $F_3$. 
 
The said subgroup of index $2$ can be defined in any $F_k, k\geq 2$, as follows:

$$
G_k:=\{f\in F_k\; | \; \log_kf'(1)\in 2\IZ\}
$$

\begin{proposition}\label{propGi}
The subgroup 
$G_{k}$ is generated by $w_i:=y_iy_{k}$ for $i=0, \ldots , k-1$
 and is
 isomorphic with the Brown-Thompson group $F_{k}$. The subgroup $G_{k}$ consists of the elements in $F_{k}$ whose normal form has even length, 
 and 
 its index in $F_{k}$ is $2$. 
\end{proposition}
\begin{proof}
First of all, we observe that the elements in $F_{k}$ whose normal form has even length form a subgroup $K$ whose index in $F_{k}$ is $2$. This subgroup coincides with $G_{k}$ because $K\leq G_{k}  < F_{k}$ and the index $[F_{k}:K]=2$.

We now follow the same strategy as in \cite[Lemma 4.7]{GS}.
Since every proper homomorphic image of $F_{k}$ is abelian \cite[Theorem 4.13]{Brown},
it is enough to exhibit a family of elements in $G_{k}$ that generates the group, satisfies the generating relations of $F_{k}$ and do not commute.
We set $w_i:=y_iy_{k}$ for $i=0, \ldots , k-1$, $w_{n}:=w_0^{-1}w_{n-k+1}w_0$ for all $n\geq k$.
Since $w_{k}=(y_0y_{k})^{-1}y_1y_{k}y_0y_{k}=y_{k}y_{3k-2}$, we have
\begin{align*}
w_1w_0& =y_1y_{k}y_0y_{k}=y_0y_{k}y_{2k-1}y_{k}=y_0y_{k}y_{k}y_{3k-2}\\
& = w_0 w_{k}
\end{align*}
For  $k-1 \geq n>i$ and $n\geq 2$, we  have
\begin{align*}
w_nw_i&=y_{n}y_{k}y_{i}y_{k} =y_{i} y_{n+k-1}y_{2k-1}y_{k} =y_{i} y_{k} y_{n+2k-2}y_{3k-2}\\
& =w_i w_{n+k-1}
\end{align*}
Now the general case follows by induction on $n+i$
\begin{align*}
w_nw_i&=w_0^{-1}w_{n-k+1}w_0w_0^{-1}w_{i-k+1}w_0=w_0^{-1}w_{n-k+1}w_{i-k+1}w_0=w_0^{-1}w_{i-k+1}w_{n}w_0\\
&=w_0^{-1}w_{i-k+1}w_0w_0^{-1}w_{n}w_0=w_iw_{n+k-1}\; .
\end{align*}
Denote by $L$ the subgroup generated by $\{w_i\}_{i\geq 0}$. Clearly, $L$ is a subgroup of $G_{k}$. We want to show that they are actually equal.

The first step is to show that $y_n^2\in L$ for all $n= 0, \ldots , k$.
Firstly, suppose that $n\leq k-1$. We have
\begin{align*}
w_nw_n & = y_ny_{k}y_n y_{k}= y_n^2 y_{2k-1}y_{k}= y_n^2 y_{k}y_{3k-2}=y_n^2 w_{k}\in L\; .
\end{align*}
and thus $y_n^2\in L$.

In the second step we show that $y_ny_{n+1}\in L$ for all $n=0, \ldots , k-1$. When $n=k-1$, we have $y_{k-1}y_{k}=w_{k-1}$, so we may suppose $n\leq k-2$. 
In this case, we have
\begin{align*}
y_ny_{n+1} &= y_ny_{k}(y_{n+1}y_{k})^{-1} y_{n+1}^2= w_n w_{n+1}^{-1} y_{n+1}^2\in L\; .
\end{align*}
We note that since $y_0^2 y_{k}^2=y_0y_1 y_0y_{k}=y_0y_1w_0 \in L$, we have $y_{k}^2\in L$.
As $y_n y_{n-1}=(y_{n-1}^{-2} y_{n-1} y_{n} y_{n}^{-2})^{-1}$, we have $y_ny_{n-1}\in L$ for all $n=1, \ldots , k$.

In the third step we show that $y_{n+k-1}y_{n+k}\in L$ for all $n=0, \ldots , k-1$.
Suppose that $n\leq k-2$. We have
\begin{align*}
w_0y_{n+k-1}y_{n+k}&= y_{0}y_{k} y_{n+k-1}y_{n+k}= (y_1y_n)(y_{n+1}y_0)\\
&= (y_1y_2 y_2^{-2} y_2y_3 y_3^{-2} \cdots y_{n-1}^{-2} y_{n-1} y_n)(y_{n+1}y_n y_n^{-2} \cdots y_1^{-2}y_1y_0)\in L\; .
\end{align*}
and so $y_{n+k-1}y_{n+k}\in L$. When $n=k-1$, the claim follows from the following equality and the fact that $y_0y_{k-1}, y_{k}y_0\in L$
\begin{align*}
y_0^2 y_{2k-2}y_{2k-1}&=(y_0y_{k-1})(y_{k} y_0)\; .
\end{align*}

In the fourth step we show that $y_{n+k}^2\in L$ for all $n=1, \ldots , k-2$.
We have
\begin{align*}
w_0y_{n+k}^2&= y_{0}y_{k} y_{n+k}^2= (y_1y_{n+1})(y_{n+1}y_0)\in L\; .
\end{align*}
and so $y_{n+k}^2\in L$.

Now, since  $y_m y_{m+1}y_0^2=y_0^2 y_{m+2k-2}y_{m+1+2k-2}$ for all $m\in \IN$, we have that $y_i y_{i+1}\in L$ for all $i\in \IN_0$.
Similarly, $y_n^2y_0^2=y_0^2 y_{n+2k-2}^2$ for all $n\geq 1$, we have that $y_i^2\in L$ for all $i\in\IN_0$.

The fifth step is to show that $y_iy_j\in L$ for $i< j$. We proceed by induction on $j-i$. If $j-i=1$ the claim was proven in the previous step. Otherwise, by induction we get
$$
y_iy_j=(y_iy_{i+1})(y_{i+1}^{-1}y_j)\in L
$$
$$
y_{i+1}^{-1}y_j=y_{i+1}^{-2}(y_{i+1}y_j)
$$
where we used that $y_{i+1}^{-2}\in L$ and $y_{i+1}y_j\in L$ by induction.

In the sixth step we show that $y_i^{\pm 1} y_j^{\pm 1}\in L$ for $i, j$. This is true since
$$
y_i^{\pm 1} y_j^{\pm 1}=(y_i^{-2\epsilon}) (y_i y_j) (y_j^{-2\delta})
$$
for suitable $\epsilon, \delta\in \{0, 1\}$. 
After all these steps, now it is clear that $L$ is equal to $G_{k}$.
\end{proof}

We now turn our attention more specifically to $\vec{F}_3\in F_3$ and will return to the general case $F_k, k\geq 2$ in the next section.
Recall from Section 1 the embedding $\iota: F_2\to F_3$. Then $\iota(G_2) = G_3 \cap \iota(F_2)$. In what follows we will keep the notation $\{y_i\}$ for the generators of $F_3$ and when necessary use the notation $\{x_i\}$ for the generators of $F_2=F$. We have $\iota(x_i)=y_{2i}, \forall i\geq 0$. 
\begin{lemma}\label{lemmay02}
It holds $\langle y_0^2, \vec{F}_3\rangle = G_3$.
\end{lemma}
\begin{proof}
Thanks to Proposition \ref{propGi} it is enough to show that $y_0y_3, y_1y_3, y_2y_3$ are in $K:=\langle y_0^2, \vec{F}_3\rangle$. By Example \ref{ese3} we know that the element $y_0y_3$ is in $\vec{F_3}$. Moreover, 
$$
y_0y_3 y_0^{-2} y_0y_3= y_1y_0 y_0^{-2} y_0y_3= y_1y_3\in K
$$
so we are left to show that $y_2y_3\in K$. To this end, recall from \cite[Lemma 3.3]{GS2} that $\langle x_0^2, \vec{F}_2\rangle = G_2$, where $G_2$ consists of the elements in $F_{2}$ whose normal form has even length.
 Since $y_2=\iota(x_1)$ and $y_0=\iota(x_0)$, we have $\iota(x_1x_0)\in \iota(G_2)$. 
Now thanks to the following equality we are done
$$
y_2y_0y_0^{-2}y_0y_3=y_2y_3 \in K.
$$ 
\end{proof}
\begin{lemma}\label{lemmay03}
Let $g\in F_3\setminus \vec{F}_3$, 
then there is a positive element in 
$(\vec{F}_3 g\vec{F}_3\cap \iota(F_2))\setminus \iota(\vec{F}_2)$.
\end{lemma}
\begin{proof}
First of all we show that there is a positive element in $\vec{F}_3 g \vec{F}_3$. Let $w$ be an element of minimal length in $\vec{F}_3 g \vec{F}_3$. 
If it is positive, we can move on to the next step. Otherwise, let $y_i$ be the last letter of $w$ in the normal form, that is $w=w' y_{i}^{-1}$. If $i=2k+1$, $wy_{2k+1}^2=w' y_{2k+1}^{-1} y_{2k+1}^2$ is an element still with normal form of minimal length, but with less negative factors. If $i=2k$, one may consider $w(y_{2k}y_{2k+2})$ which does the job. By iteration we 
 get a positive element in $\vec{F}_3 g \vec{F}_3$.

We now show that it is actually possible to find a positive element in 
$(\vec{F}_3 g\vec{F}_3\cap \iota(F_2))\setminus \iota(\vec{F}_2)$. This means that it is the product of elements in the set $\{y_{2i}\}_{i\geq 0}$. So far we have an element  $w$ which can be expressed as $y_{i_1}y_{i_2}\cdots y_{i_n}$ in its  normal form. If $i_1=0$ we may consider $y_3^2(y_3^{-1}y_0^{-1})w=y_3^2y_{i_2}\cdots y_{i_n}$, which has the same length and less factors equal to $y_0$. If $i_2=0$, after using the defining relations and the former trick we may erase another factor equal to $y_0$. After finitely many steps, we may suppose that there are no factors equal to $y_0$.  
Now that we know that $w$ does not contain $y_0$, we may use the relation $(y_0y_2)^{-1} w (y_0y_2)=\varphi_R^2(w)$ (as defined in Lemma \ref{lemmashift}). This allows us to suppose that all the indices are non-zero and arbitrarily big. 
Suppose that $i_1=2k$, by multiplying $w$ (repeatedly) on the left by factors of the form $y_{2h+2}^{-1}y_{2h}^{-1}\in\vec{F}_3$ we get 
\begin{align*}
& y_{2k+2m+3}y_{2k+2m+2}y_{2k+2m+2}^{-1}y_{2k+2m}^{-1}\cdots y_{2k+2}^{-1}y_{2k}^{-1} y_{2k} y_{i_2}\cdots y_{i_n} = \\
&= y_{2k+2m+3} y_{i_2}\cdots y_{i_n} =\\
&=  y_{i_2}\cdots y_{i_n} y_{2k+2m+3-2(n-1)}
\end{align*}
where we used that $y_{2k+2m+3}y_{2k+2m+2}=y_{2k+2m+2}y_{2k+2m+5}\in\vec{F}_3$.
After multiplying by another suitable element we get an element 
of the form
\begin{align*}
&  y_{i_2}\cdots y_{i_n} y_{2k+2m-2(n-1)+3} (y_{2k+2m-2(n-1)+3}^{-1} y_{2k+2m-2(n-1)}^{-1}) (y_{2k+2m-2(n-1)}y_{2k+2m-2(n-1)+2})\\
&=  y_{i_2}\cdots y_{i_n} y_{2k+2m-2n+2}
\end{align*}
This allows us to suppose that the first factor in $w$ is odd, that is 
$i_1=2k+1$, with $k\neq 0$.
In this case just consider
\begin{align*}
& (y_{2k-2}y_{2k+1}) y_{2k+1}^{-2}(y_{2k+1}y_{i_2}\cdots y_{i_n}) \\
&= y_{2k-2}y_{i_2}\cdots y_{i_n} 
\end{align*}
which has the same length and less odd factors (i.e. $y_{2h+1}$).
If $i_1=1$, just conjugate by $y_0y_2$ so that we may assume that $i_1\geq 3$. Now after repeating the previous arguments, we get an element without odd factors.

We make a final observation. We started from $g\in F_3\setminus \vec{F}_3$ and we considered an element 
$w$ of minimal length in $\vec{F}_3g\vec{F}_3$. In all the steps of the proof we multiplied $w$ by suitable elements of $\vec{F}_3$ and got an element $w'=f_1 w f_2\in (\vec{F}_3 g\vec{F}_3\cap \iota(F_2))$, where $f_1, f_2$ are some suitable elements in $\vec{F}_3$. For this reason it is clear that the resulting element cannot be in $\vec{F}_3$ (and in $\iota(\vec{F}_2)\subset \vec{F}_3$).
\end{proof}

\begin{theorem}
Let  $g$ be an element of $F_3\setminus \vec{F}_3$. Then, the group generated by $g$ and $\vec{F}_3$ is $G_3$ if $g$ has even length, and $F_3$ otherwise.
In particular, the group $\vec{F}_3$ is maximal in $G_3$.
\end{theorem} 
\begin{proof}
Let  $g$ be an element of $F_3\setminus \vec{F}_3$.  Thanks to the previous lemma we know that 
$(\vec{F}_3 g\vec{F}_3\cap \iota(F_2))\setminus \iota(\vec{F}_2)$
contains a positive element. Without loss of generality, we may suppose that $g$ is positive and in $\iota(F_2)\setminus \iota(\vec{F}_2)$.
Therefore, we have to consider two cases, depending on whether $g$ is in $\iota(G_2)$ or not.
If $g\in \iota(G_2)\setminus \iota(\vec{F}_2)$, by   \cite[Theorem 3.12]{GS2} we have that $\langle g, \iota(\vec{F}_2)\rangle =\iota(G_2)$. In particular, we have $y_0^2\in \langle g, \vec{F}_3\rangle$ and thus $\langle g, \vec{F}_3\rangle = G_3$ by Lemma \ref{lemmay02}.

Suppose that $g\in \iota(F_2)\setminus \iota(G_2)$. This means that its normal form has odd length.  By  \cite[Theorem 3.12]{GS2} we  have that  $\langle g, \iota(\vec{F}_2)\rangle = \iota(F_2)$. 
In this case, we clearly have that $y_0^2\in \langle g, \vec{F}_3\rangle$ and, hence, $G_3\subset\langle g, \vec{F}_3\rangle$. Now $\langle g, G_3\rangle$ must be equal to $F_3$ by Proposition
\ref{propGi}.
\end{proof}
\begin{corollary}
The Brown-Thompson group $F_3$ has a maximal subgroup of infinite index  isomorphic to $\vec{F}_3$ that does not stabilize any $x\in (0,1)$.
\end{corollary}
\begin{proof}
Recall from Proposition \ref{propGi} the notation $w_0$, $w_1$, $w_2$ for the three generators of $G_3$. Denote by $\Phi: G_3\to F_3$ the isomorphism established in that Proposition. Since the action of 
$y_0$ on the unit interval $[0,1]$ was described in Example \ref{exe2}, it is clear that  $y_0=\Phi(w_0)\in \Phi(\vec{F}_3)\leq \Phi(G)= F_3$ provides an element which does not fix any $x\in (0,1)$.
\end{proof}

We are at last in a position to give a proof of Theorem \ref{theo-stab}.
\begin{proof}[Proof of Theorem \ref{theo-stab}]
 It can be easily seen that all the generators of $\vec{F}_3$  preserve the subset $ Z$ and, therefore, $\vec{F}_3\leq {\rm Stab}( Z)$. Since neither $y_0$, nor $y_0^2$ belong to ${\rm Stab}( Z)$, it holds $\vec{F}_3={\rm Stab}( Z)$.
\end{proof}

Similarly to the case of $\vec F\in F$ (compare with \cite[Corollary 3]{GS} and \cite[Section 5.2]{jo2}), the above results allow us to conclude that the quasi-regular representation of $F_3$ associated with $\vec{F}_3$ is irreducible.
 \begin{corollary}
The ternary oriented Thompson group $\vec F_3$ coincides with its commensurator. In particular, the quasi-regular representation of $F_3$ on $F_3/\vec{F}_3$ is irreducible.
\end{corollary}
 \begin{proof}
 First of all, we point out that for any $f, g\in F_3$, $x\in [0,1]$,  we adopt the standard notation $f\cdot g(x)\equiv g(f(x))$.\\
 It is enough to prove the claim about the commensurator since the irreducibility of the representation then  follows from \cite{Ma}.
 We follow the same strategy as \cite{GS}. 
 Let $h\in F_3\setminus \vec{F}_3$ and $I:=[\vec{F}_3:\vec{F}_3\cap h\vec{F}_3h^{-1}]$. Suppose that $I<\infty$ and pick $g=(y_0y_3)^{-1}\in\vec{F}_3$. Then, there exists an $r\in\IN$ such that $g^r\in h\vec{F}_3h^{-1}$, or equivalently $h^{-1}g^rh\in\vec{F}_3$. This implies that $h^{-1}g^{rk}h\in\vec{F}_3$ for all $k\in \IN$. We will reach a contradiction by showing that  $h^{-1}g^{n}h\not\in\vec{F}_3$ for all $n$ large enough.
 
 First of all, we observe that the same argument as in \cite[Lemma 4.14]{GS} shows that there exists an $m\in\IN$ such that 
 for any finite ternary fraction $t<3^{-m}$, the weight $c(t)$ is equal to that of $h(t)$.
 Secondly, 
 thanks to the Example \ref{ese3} it is easy to see that 
 for every $t\neq 1$ we have that $g^n(t)< 3^{-m}$ if $n$ is large enough.
 
 Since $h\not\in\vec{F}_3$, there exists a $t\in Z$ such that $t_1:=h^{-1}(t)\not\in  Z$.
 We observe that for all $l\in\IN$ we have $g^l(t_1)\not\in  Z$. As $g^n(t_1)<3^{-m}$ we have that   $h(g^n(t_1))\not\in Z$.
Therefore, it holds $h^{-1}g^nh(t)=h(g^n(t_1))\not\in Z$.
Since $h^{-1}g^nh$ does not stabilize $ Z$, we have that $h^{-1}g^nh\not\in\vec{F}_3$ and we are done.
 \end{proof} 
Given a graph $G$, its chromatic polynomial   ${\rm Chr}_G(x)$ is the unique polynomial which, when evaluated at a $Q\in\IN$, gives the number of proper vertex colourings of $G$ with $Q$ colours. The function
  $$
  {\rm Chr}(g,Q):=\frac{ {\rm Chr}_{\Gamma(T_+,T_-)}(Q)}{(Q-1)^{n-1}} \qquad Q\in\IN_{\geq 2}:=\{2,3,4,\ldots\}\; ,
  $$
  where $n$ is the number of the leaves of $T_\pm$, is  a positive type function on $F_3$ (this can be shown by the same argument as was used in \cite{AiCo1} for $F_2$).  The quasi-regular representation of $F_3$ associated with $\vec{F}_3$ coincides  with the representation associated to this positive type function with $Q=2$.

\section{On maximal subgroups of $F_k$}

 Given a point $x\in (0,1)$,  and $k\geq 2$, the stabilizer Stab$_{F_k}(x)$ 
  under the natural action of the elements of $F_k$ seen as homeomorphisms of the unit interval are called the {\it parabolic subgroups} of $F_k$. These subgroups are natural examples of maximal subgroups of infinite index in $F_k$. 
 This was proven by Savchuk for $F_2$ and the proof  \cite[Prop. 2.4, p. 360]{Sav} readily adapts to any $k\geq 2$.
 \begin{theorem}\label{theo4}
 For any $x\in (0,1)$, the stabilizers Stab$_{F_k}(x)$ are maximal subgroups of infinite index of the Brown-Thompson group $F_k$. Moreover, the associated quasi-regular representations are irreducible.
 \end{theorem} 
For the second statement see  \cite[Lemma 8]{Gar}.

It is therefore natural to inquire, what are non-parabolic maximal subgroups of infinite index in $F_k, k\geq 2$, that would in some way generalise the subgroups $\vec{F}_2$ and $\vec{F}_3$ respectively in $F_2$ and in $F_3$.

The aim of this section is to exhibit subgroups of the Brown-Thompson groups $F_k$, which  in some way generalise the oriented subgroup $\vec{F}_2$.

Denote by $\CT_k$ the set of $k$-ary planar rooted trees. 
For every $k\geq 2$ there exists a natural map $\Phi_{k} : \CT_{2k-1}\to\CT_{k}$ displayed in Figure \ref{map}. 
  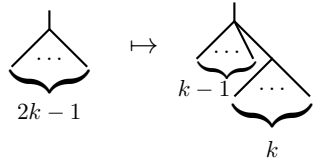
\begin{figure}
 \caption{The map $\Phi_k: \CT_{2k-1}\to \CT_{k}$ is obtained by replacing every  vertex of degree $2k$ with the tree displayed below.}\label{map}
\phantom{This text will be invisible} 
\[\begin{tikzpicture}[x=1cm, y=1cm,
    every edge/.style={
        draw,
      postaction={decorate,
                    decoration={markings}
                   }
        }
]

\draw[thick] (0,0)--(.5,.5)--(1,0);
\draw[thick] (0.5,0.75)--(.5,.5);
\node at (0,-1.2) {$\;$};
\node at (1.75,0.25) {$\scalebox{1}{$\mapsto$}$};

\node at (0.5,.1) {$\scalebox{.8}{$\ldots$}$};
\node at (0.5,-.6) {$\scalebox{.8}{$2k-1$}$};
\node at (0.5,-.2) {$\scalebox{1.5}{$\underbrace{	\; }$} $};

\end{tikzpicture}
\begin{tikzpicture}[x=1cm, y=1cm,
    every edge/.style={
        draw,
      postaction={decorate,
                    decoration={markings}
                   }
        }
]

\draw[thick] (0.5,0.75)--(.5,.5);
\draw[thick] (0.75,0)--(.5,.5);
 \draw[thick] (0,0)--(.5,.5)--(1,0);
 \draw[thick] (.5,-.5)--(1,0)--(1.5,-0.5);

\node at (0.4,.1) {$\scalebox{.8}{$\ldots$}$};
\node at (1,-.4) {$\scalebox{.8}{$\ldots$}$};

\node at (0.1,-.4) {$\scalebox{.8}{$k-1$}$};
\node at (0.4,-.15) {$\scalebox{1.1}{$\underbrace{	\; }$} $};

\node at (1,-1.2) {$\scalebox{.8}{$k$}$};
\node at (1,-.75) {$\scalebox{1.5}{$\underbrace{	\; }$} $};

\node at (0,-1.2) {$\;$};
\end{tikzpicture}
\]
\end{figure}
For every $k$, the map $\Phi_k$     induces a monomorphism $\varphi_k : F_{2k-1}\to F_{k}$. We denote by $\{y_i\}_{i\geq 0}$ and $\{z_i\}_{i\geq 0}$ the canonical generators of $F_{k}$ and $F_{2k-1}$, respectively. 
It is easy to see that $\varphi_k(z_i)=y_iy_{i+k-1}$. 
We define the following subgroups

$$H_{k}:=\varphi_i(F_{2k-1})=\langle y_iy_{i+k-1}\; , \; i\geq 0\rangle\leq F_{k}$$

It is clear that for  all $k$, the subgroup $H_{k}$ sits inside $G_{k}$,  by observing that $y_i'(1)=k^{-1}$ for all $i$), see Figure \ref{genFk}.
We observe that for $k=2$, the subgroup $H_2$ is exactly the  oriented subgroup $\vec{F}$. 
For $k=3$, the subgroup 
$H_3$ does not coincide with $\vec{F}_3$ since $\varphi_3(z_1)=y_1y_3\not\in \vec{F}_3$. 
\begin{figure}
\caption{The generators of $F_k$. 
 In the tree diagram of $y_i$, with $i=1, \ldots, k-2$,  a split is attached on the of $i$-th edge below the root of the top tree.
}\label{genFk}
\phantom{This text will be invisible} 
\[
\begin{tikzpicture}[x=.35cm, y=.35cm,
    every edge/.style={
        draw,
      postaction={decorate,
                    decoration={markings}
                   }
        }
]

\node at (-1.5,0) {$\scalebox{1}{$y_0=$}$};
\node at (.55,0) {$\scalebox{.65}{$\ldots$}$};
\node at (2.55,0) {$\scalebox{.65}{$\ldots$}$};
\node at (-1.25,-3) {\;};

 \draw[thick] (2,2)--(2,2.5);
 \draw[thick] (2,-2)--(2,-2.5);
 
\draw[thick] (0,0) -- (2,2)--(4,0)--(2,-2)--(0,0);
\draw[thick] (1,1) -- (1,0)--(2,-2);
\draw[thick] (1,1) -- (2,0)--(3,-1);
\draw[thick] (2,2) -- (3,0)--(3,-1);

\end{tikzpicture}
\;\;
\begin{tikzpicture}[x=.35cm, y=.35cm,
    every edge/.style={
        draw,
      postaction={decorate,
                    decoration={markings}
                   }
        }
]

 \draw[thick] (1,3)--(1,3.5);
 \draw[thick] (1,-3)--(1,-3.5);

\node at (-4.25,0) {$\scalebox{1}{$y_{k-1}=$}$};
\node at (-1.25,-3.25) {\;};

\node at (-1.45,0) {$\scalebox{.65}{$\ldots$}$};
\node at (.55,0) {$\scalebox{.65}{$\ldots$}$};
\node at (2.55,0) {$\scalebox{.65}{$\ldots$}$};

\draw[thick] (2,2)--(1,3)--(-2,0)--(1,-3)--(2,-2);

\draw[thick] (0,0) -- (2,2)--(4,0)--(2,-2)--(0,0);
 \draw[thick] (1,1) -- (2,0)--(3,-1);

\draw[thick] (0,0) -- (2,2)--(4,0)--(2,-2)--(0,0);
\draw[thick] (1,1) -- (1,0)--(2,-2);
\draw[thick] (1,1) -- (2,0)--(3,-1);
\draw[thick] (2,2) -- (3,0)--(3,-1);

\draw[thick] (1,3) -- (-1,0)--(1,-3);

\end{tikzpicture}\;\;
\begin{tikzpicture}[x=.35cm, y=.35cm,
    every edge/.style={
        draw,
      postaction={decorate,
                    decoration={markings}
                   }
        }
]

\node at (-1.5,0) {$\scalebox{1}{$y_i=$}$};
\node at (-1.25,-3.25) {\;};

\node at (1.75,1.15) {$\scalebox{.75}{$i$}$};
\node at (1.35,0) {$\scalebox{.75}{$i$}$};

 \draw[thick] (2,2)--(2,2.5);
 \draw[thick] (2,-2)--(2,-2.5);
 
\node at (.55,0) {$\scalebox{.65}{$\ldots$}$};
\node at (2.55,0) {$\scalebox{.65}{$\ldots$}$};
\node at (3.55,0) {$\scalebox{.65}{$\ldots$}$};

\node at (8.55,0) {$\scalebox{1}{$i=1,\ldots , k-2.$}$};


\draw[thick] (0,0) -- (2,2)--(4,0)--(2,-2)--(0,0);
 \draw[thick] (1,0)--(2,-2);
\draw[thick] (3,0)--(2,1) -- (1,0); 
\draw[thick] (2,0)--(3,-1);
\draw[thick] (2,2) -- (2,0);
\draw[thick] (3,0)--(3,-1);


\end{tikzpicture}
\]
\end{figure}
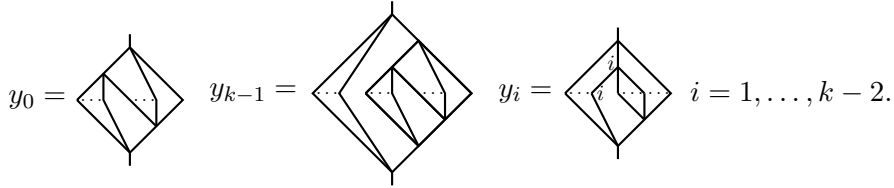

\begin{proposition}\label{stabHi}
Let  $\alpha_k : G_k\to F_k\equiv\langle z_i\rangle$ be the isomorphism mapping $y_iy_k$ to $z_i$ for $i=0, \ldots, k-1$ (see the proof of Proposition \ref{propGi}). Then, the subgroup $\alpha_k(H_k)$ does not stabilise any $x\in (0,1)$.
\end{proposition}
\begin{proof}
As in Proposition \ref{propGi}, set $w_i:=y_iy_k$ for $i=0, \ldots , k-1$ and $w_{n}:=w_0^{-1}w_{n-k+1}w_0$ for all $n\geq k$. Since
$$
w_0w_{k-1}w_k^{-1}=y_0y_{k-1}\in H_k
$$
we have that $\alpha_k(y_0y_{k-1})=\alpha_k(w_0w_{k-1}w_k^{-1})=z_0z_{k-1}z_k^{-1}$.
Now it is easy to see that this element does not have fixed points in $(0,1)$ and we are done.
\end{proof}

The following natural questions about the subgroups $H_k$ remain open for now.
\begin{problem}
Is the index of $H_k$ in $G_k$ infinite?
Is $H_{k}$ a maximal subgroup of $G_{k}$?
\end{problem}
\begin{remark}
In \cite{GS2}, implicit examples of infinite index maximal subgroups of $F$ were exhibited
using the theory of diagram groups. 
Since all Brown-Thompson groups are diagram groups,
it is possible that   similar methods 
could be applied to them as well. We haven't explored this direction.
\end{remark}
For $F$ there is another possible source of maximal infinite index subgroups of $F$ coming from a natural generalisation of the map $\varphi_2$ defined above. 
Let $T$ be a rooted planar binary tree with $k$ leaves.
As usual we put its leaves on the $x$-axis, precisely on the non-negative integers $0, 1, 2, \ldots, k-1$.
We denote by $\ell_T(0)$ the number of left edges in the path from the left-most leaf to the root and by $\ell_T(k-1)$ the number of right edges in the path from the right-most leaf to the root.
Now we define an injective map $\alpha_T: F_k \to F$:
given a tree diagram $(T_+,T_-)\in F_k$, replace any vertex of degree $k+1$ with the tree $T$, \cite{Ren}.  

We recall that   the projection $\pi: F\to F/[F,F]=\mathbb{Z}\oplus \mathbb{Z}$ can be described as $\pi(f)=(\log_2 f'(0),\log_2 f'(1))$, see \cite{CFP}. 
If $f$ is represented by a pair of trees $(T_+,T_-)$, then $\log_2 f'(0)$ is equal to the number of left edges in the path from the left-most leaf to the root of $T_+$ minus the same number for $T_-$.
Similarly, $\log_2 f'(1)$ is equal to the number of right edges in the path from the right-most leaf to the root of $T_+$ minus the same number for $T_-$.

Recall from \cite{BW} that, for any $a$, $b\in \mathbb{N}$, one can define the rectangular subgroups of $F$ as
$$
K_{(a,b)}:=\{f\in F\; | \;  \log_2 f'(0)\in a \mathbb{Z},  \log_2 f'(1)\in b \mathbb{Z} \}
$$
All these subgroups have finite index and are isomorphic with $F$.   
\begin{proposition}
The subgroup $\alpha_T(F_k)$ sits inside of $K_{(\ell_T(0),\ell_T(k-1))}$.
\end{proposition}
\begin{proof}
By looking at the tree diagrams of the generators $y_0$, \ldots , $y_{k-1}$ of $F_k$ (see 
Figure \ref{genFk}), one can see that 
\begin{align*}
\pi(\alpha_T(y_0)'(0),\alpha_T(y_0)'(1)) & =(2\ell_T(0)-\ell_T(0),\ell_T(k-1)-2\ell_T(k-1))\\
& =(\ell_T(0),\ell_T(k-1))\\
\pi(\alpha_T(y_i)'(0),\alpha_T(y_i)'(1)) &=(\ell_T(0)-\ell_T(0),\ell_T(k-1)-2\ell_T(k-1))\\
& =(0,\ell_T(n-1))\qquad i=1, \ldots , n-2\\
\pi(\alpha_T(y_{k-1})'(0),\alpha_T(y_{k-1})'(1))& =(\ell_T(0)-\ell_T(0),2\ell_T(k-1)-3\ell_T(k-1))\\
&=(0,-\ell_T(k-1))\; .
\end{align*} 
\end{proof}
Note that for $a=1$, $b=2$, we get $K_{(1,2)}=G_2$, also $\varphi_2$ is $\alpha_T$ with $T$ the binary tree with three leaves where the caret is to the right from the root. 
\begin{problem}
What is the index of $\alpha_T(F_k)$ in $K_{(\ell_T(0),\ell_T(k-1))}$?
Is $\alpha_T(F_k)$ maximal in $K_{(\ell_T(0),\ell_T(k-1))}$?
\end{problem}

\section*{Acknowledgements} 
 T.N. acknowledges support of  Swiss NSF grants 200020-178828 and 200020-200400.


\begin{thebibliography}{99}

\bibitem{A} V. Aiello, \emph{On the Alexander Theorem for the oriented Thompson group $\vec{F}$}, Algebraic \& Geometric Topology, 20 (2020) 429--438, preprint arXiv:1811.08323 (2018).

  
\bibitem{ABC} V. Aiello, A. Brothier, R. Conti. \emph{Jones representations of Thompson's group $F$ arising from Temperley-Lieb-Jones algebras}, Int. Math. Res. Not. 15 (2021), 11209--11245, preprint arXiv:1901.10597 (2019).

\bibitem{AiCo1} V. Aiello, R. Conti, \emph{Graph polynomials and link invariants as positive type functions on Thompson's group $F$}, 
 J. Knot Theory Ramifications 28 (2019), no. 2, 1950006, 17 pp.
doi: 10.1142/S0218216519500068 , preprint arXiv:1510.04428

\bibitem{AiCo2} V. Aiello, R. Conti, \emph{The Jones polynomial and functions of positive type on the oriented Jones-Thompson groups $\vec{F}$ and $\vec{T}$},  Complex Anal. Oper. Theory (2019) 13: 3127. doi: 10.1007/s11785-018-0866-6 ,  preprint arXiv:1603.03946 
\bibitem{ACJ} V. Aiello, R. Conti, V.F.R. Jones, \emph{The Homflypt polynomial and the oriented Thompson group}, Quantum Topol. 9 (2018), 461-472. preprint arXiv:1609.02484


\bibitem{AJ} V. Aiello, V.F.R. Jones, \emph{On spectral measures for certain unitary representations of R. Thompson's group F}. J. Funct. Anal., Volume 280, Issue 1, 1 January 2021, 108777, preprint arXiv:1905.05806 (2019).

%
 
\bibitem{B} J. Belk, \emph{Thompson's group F}. Ph.D. Thesis (Cornell University).  preprint arXiv:0708.3609 (2007).

\bibitem{BM} J. Belk, F. Matucci. \emph{Conjugacy and dynamics in Thompson's groups}. Geometriae Dedicata 169.1 (2014): 239-261.

\bibitem{BW} C. Bleak, B. Wassink. \emph{Finite index subgroups of R. Thompson's group F.} preprint arXiv:0711.1014 (2007).

\bibitem{BJ} A. Brothier, V.F.R. Jones. \emph{Pythagorean representations of Thompson's groups}. Journal of Functional Analysis, Vol. 277 (2019) 2442-2469.
 


\bibitem{Brown} K. S. Brown, \emph{Finiteness properties of groups}. Proceedings of the Northwestern conference on cohomology of groups (Evanston, Ill., 1985). J. Pure Appl. Algebra 44 (1987), no. 1--3, 45--75.

\bibitem{CFP}
J.W. Cannon, W.J. Floyd,   W.R. Parry, 
Introductory notes on Richard Thompson's groups.
{\em L'Enseignement  Math\'ematique} 
{\bf 42} (1996): 215--256
 

\bibitem{Gar} \L{}. Garncarek, \emph{Analogs of Principal Series Representations for Thompson's Groups F and T.} Indiana University Mathematics Journal, vol. 61, no. 2, 2012, pp. 619--626

\bibitem{GS} G. Golan, M. Sapir, \emph{On Jones' subgroup of R. Thompson group $F$},  Journal of Algebra 470 (2017), 122-159.

\bibitem{GS2} G. Golan, M. Sapir, \emph{On subgroups of R. Thompson's group $F$}, Transactions of the American Mathematical Society 369.12 (2017): 8857--8878%
%
\bibitem{jo2}V.F.R. Jones, \emph{Planar Algebras I}. preprint arXiv:	math/9909027 (1999).
%
%
%
%
%
\bibitem{Jo14} V.F.R. Jones, \emph{Some unitary representations of Thompson's groups $F$ and $T$}. J. Comb. Algebra {\bf 1} (2017), 1--44.
%
\bibitem{Jo16}V.F.R. Jones, \emph{A no-go theorem for the continuum  limit of a quantum spin chain}. \emph{Comm. Math. Phys.} {\bf 357} (2018), 295--317.
%
\bibitem{Jo18} V.F.R. Jones, \emph{On the construction of knots and links from Thompson's groups}.  In: Adams C. et al. (eds) Knots, Low-Dimensional Topology and Applications. KNOTS16 2016. Springer Proceedings in Mathematics \& Statistics, vol 284. Springer, Cham.
preprint arXiv:1810.06034 (2019).

\bibitem{Jo19} V.F.R. Jones, \emph{Irreducibility of the Wysiwyg representations of Thompson's groups}. preprint arXiv:1906.09619 (2019).

\bibitem{Ma} G. Mackey, \emph{The theory of unitary group representations}, The University of Chicago Press (1976).
%
\bibitem{Ren} Y. Ren, \emph{From skein theory to presentations for Thompson group}.  Journal of Algebra, {\bf 498}, 178--196 (2018).
%

\bibitem{Sav2} D. Savchuk,  \emph{Some graphs related to Thompson's group $F$}. Combinatorial and geometric group theory, 279--296, Trends Math., Birkh\"auser/Springer Basel AG, Basel, 2010.

\bibitem{Sav} D. Savchuk, (2015). \emph{Schreier graphs of actions of Thompson's group $F$ on the unit interval and on the Cantor set}. Geometriae Dedicata, 175(1), 355-372.

 

\end{thebibliography}
\end{document}